\documentclass[
	paper=A4,
	pagesize=auto,
	fontsize=12pt
]{scrartcl}
\usepackage{libertine}
\recalctypearea

\usepackage[british]{babel}
\usepackage[style=alphabetic,url=false,isbn=false,doi=true,eprint=false,backend=biber]{biblatex}
\AtEveryBibitem{\clearlist{language}}
\usepackage{graphicx}
\usepackage{amsmath,amssymb,amsthm}
\usepackage{csquotes}
\usepackage{bm}
\usepackage{enumerate}

\usepackage{hyperref}
\usepackage{xurl}
\hypersetup{colorlinks,breaklinks=true}

\usepackage{cleveref}

\theoremstyle{plain}
\newtheorem{theorem}{Theorem}[section]
\newtheorem{proposition}[theorem]{Proposition}
\newtheorem{lemma}[theorem]{Lemma}
\newtheorem{corollary}[theorem]{Corollary}

\theoremstyle{remark}

\newtheorem{remark}[theorem]{Remark}

\theoremstyle{definition}
\newtheorem{definition}[theorem]{Definition}

\newlength{\JZHeightOfX}
\newcommand{\JZOrcidlink}[1]{
\setlength{\JZHeightOfX}{\fontcharht\font`X}
\includegraphics[height=\JZHeightOfX]{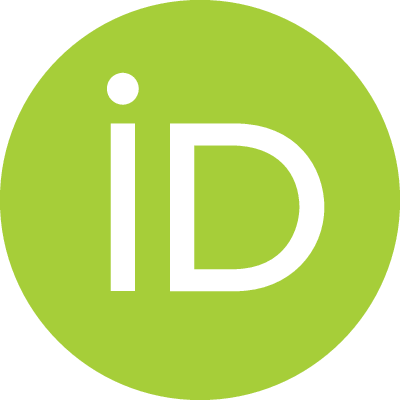}
\href{https://orcid.org/#1}{#1}
}

\begin{document}
\title{Wetterich's Equation and its Boundary Conditions for Radon Measures on Locally Convex Spaces}
\author{
Jobst Ziebell\hspace{2em}\JZOrcidlink{0000-0002-9715-6356}\\
\small{Faculty of Mathematics and Computer Science, Friedrich-Schiller-University, Jena, Germany}
}
\date{\today}
\maketitle

\begin{abstract}
Wetterich's equation and corresponding flows of effective average actions are used frequently in theoretical physics to study the properties of quantum field theories.
Under appropriate conditions, Wetterich's equation also holds for Radon measures on locally convex spaces and the domain of the effective average action is the Lusin affine kernel of the measure.
The resulting flow interpolates between the convex conjugate of the cumulant-generating function of the measure in question and its (generalised) Onsager-Machlup function.
The underlying metric of the latter is induced by a family of measurable bilinear functionals that can be understood as bilinear versions of Lusin measurable linear functionals.
\end{abstract}

\section{Introduction}
In theoretical physics, one studies the flow of \enquote{effective average actions} which in the case of a real-valued scalar field are real-valued functions $\Gamma_k$ on an affine space \cite{src:Flörchingher:FunctionalInformationGeometry}.
The asymptotic boundary condition for large flow parameters should correspond to the \enquote{classical action}.
In \cite{src:Ziebell:RigorousFRG}, the author showed that this picture indeed holds for certain regularised models, encoded as measures that are (measure-theoretically) equivalent to certain Gaußian measures.
Then the asymptotic boundary condition becomes the sum of the squared Cameron-Martin norm and the logarithm of the given density.
It is an intriguing question, whether the flow equation also holds for non-regularised quantum field theories and how the boundary conditions behave.
From a purely mathematical point of view, one can think of taking a Radon measure on a locally convex space and use Wetterich's equation as a tool to calculate its Onsager-Machlup function.
The natural domain of the effective average action then becomes the Lusin affine kernel of the measure.

While not touched upon in this paper, the reverse programme is also of considerable interest:
Given some function on an affine space, can it be the Onsager-Machlup function of a measure?
In theretical physics, this is exactly the problem of quantising a given (Euclidean) classical action functional.
It is well-established that this is possible in e.g. two spacetime dimensions for a wide class of actions \cite{src:Simon:PPhi2} and a landmark result also produced the celebrated $\Phi^4_3$ model (see e.g. \cite{src:GlimmJaffe:QuantumPhysics} for a short summary and a list of original references).
At the same time, it is known - though from a completely different quantisation ansatz - that many action functionals can only correspond to \enquote{trivial} quantum field theories, that is to Gaußian measures \cite[Section 21.6]{src:GlimmJaffe:QuantumPhysics}.
Ideally, the existence of a solution to the flow equation with given boundary conditions should be able to answer the question of the triviality of a given action functional.

A final motivation of this topic, also not touched upon in this paper, is the necessity of \enquote{renormalisation}, which in this setting can be summarised as follows:
Suppose a sequence $(\mu_n)_{n \in \mathbb{N}}$ of Radon probability measures are given that satisfy the flow equation with boundary conditions given by a suitable sequence $(f_n)_{n \in \mathbb{N}}$ of functions.
If $\mu_n$ converges weakly (or perhaps in some other sense) to a Radon measure $\mu$, in what sense do the boundary conditions $f_n$ converge?
From a physical point of view, $\lim_{n \to \infty} f_n$ is typically ill-defined, thanks to the divergence of certain parameters (coupling constants) in $f_n$.
However, the hope of many theoretical physicist - including the author - is that the flow $(\Gamma_k)_{k \ge 0}$ of the effective average action of the limit measure $\mu$ captures these divergences by making the boundary condition
\begin{equation}
\lim_{k \to \infty} \Gamma_k \left( \frac{y}{k} \right)
\end{equation}
well-defined (the factor $1/k$ comes from power counting but could in general be something else).
That is, even if the boundary condition $\lim_{n \to \infty} f_n$ is ill-defined, there should be a way to capture the asymptotics of it.
In fact, one may hope that for the corresponding flows $(\Gamma^n_k)_{k \ge 0}$ of the measures $\mu_n$ respectively, at finite $k \ge 0$
\begin{equation}
\lim_{n \to \infty} \Gamma^n_k
=
\Gamma_k
\end{equation}
in a suitable sense.
\section{Preliminaries}
\label{sec:Preliminaries}
$\overline{\mathbb{R}} = \mathbb{R} \cup \{ - \infty \} \cup \{ \infty \}$ and for a function $f : I \to \mathbb{R}$ on an interval $I \subseteq \mathbb{R}$, denote by $D^+ f(x) = \limsup_{y \searrow x} [f(y) - f(x)]/(y - x) \in \overline{\mathbb{R}}$ the upper right Dini derivative at a point $x \in I$ with $x < \sup I$.
Similarly, $D_- f(x)$ denotes the lower left Dini derivative at $x$.
\begin{lemma}
\label{lem:FiniteDiniImpliesAbsoluteContinuity}
Let $a < b$ and $f : [a, b] \to \mathbb{R}$ be a continuous, monotonically increasing function with $D^+ f(x) \in \mathbb{R}$ for every $x \in [a, b)$.
Then $f$ is absolutely continuous.
\end{lemma}
\begin{proof}
Let $\lambda$ denote the Lebesgue measure on $[a, b]$.
From \cite[p. 272]{src:Saks:TheoryOfTheIntegral}, $D^+ f$ is $\lambda$-measurable and we have
\begin{equation}
\lambda \left( f \left( B \right) \right) \le \int_B \left| D^+ f \right| \mathrm{d} \lambda \, ,
\end{equation}
for every $\lambda$-measurable set $B \subseteq [a, b]$.
In particular, it follows that $f$ takes $\lambda$-null sets to $\lambda$-null sets such that the claim follows from \cite[Theorem 7.18]{src:Rudin:RealAndComplexAnalysis}.
\end{proof}
All vector spaces in this paper are considered over the field of real numbers and the topologies considered are \textbf{not} assumed to be Hausdorff unless explicitly stated.
The complexification of a vector space $X$ is denoted by $X_{\mathbb{C}}$.
If $X$ is equipped with some topology $\tau$ (not necessarily a vector topology), we define $(X, \tau)^*$ to be the set of all linear real-valued functionals on $X$ that are \textbf{upper semicontinuous} with respect to $\tau$.
It is easy to see that $(X, \tau)^*$ is a \textbf{cone}, i.e. it is closed under linear combinations with nonnegative coefficients.
Note that in the case of a vector topology, $(X, \tau)^*$ is precisely the usual topological dual and a vector space in its own right.
We also define the \textbf{conjugate topology}
\begin{equation}
\bar{\tau} = \left\{ -U : U \in \tau \right\}
\end{equation}
and it is straightforward to see that $(X, \tau)^* = - (X, \bar{\tau})^*$.
We shall consider \textbf{asymmetrically normed spaces} which are pairs $(X, p)$ where $X$ is a vector space and $p$ is an \textbf{asymmetric norm} on $X$, i.e. a function $p : X \to [0, \infty)$ such that
\begin{equation}
p \left( x \right) = p \left( -x \right) = 0 \iff x = 0 \, ,
\qquad
p \left( r x \right) = r p \left( x \right) \, ,
\qquad
p \left( x + y \right) \le p \left( x \right) + p \left( y \right)
\end{equation}
for all $x, y \in X$ and $r \ge 0$.
The asymmetric norm induces a first-countable (not necessarily Hausdorff) topology for which sets of the form $x + p^{-1}([0, \epsilon ))$ for $x \in X$ and $\epsilon > 0$ form a neighbourhood basis at any given $x \in X$.
To such a pair $(X, p)$, we associate the \textbf{conjugate} asymmetrically normed space $(X, \bar{p})$ with $\bar{p}(x) = p(-x)$ and remark that the induced topology is indeed the corresponding conjugate topology.
We also equip $(X, p)^*$ with the functional $p^* : (X, p)^* \to [0, \infty)$ given by $\phi \mapsto \sup \{ \phi(x) | x \in X : p(x) \le 1 \}$.
$p^*$ is a \textbf{cone norm} in the sense that
\begin{equation}
p^* \left( \phi \right) = 0 \iff \phi = 0 \, ,
\qquad
p^* \left( r \phi \right) = r p^* \left( \phi \right) \, ,
\qquad
p^* \left( \phi + \psi \right) \le p^* \left( \phi \right) + p^* \left( \psi \right)
\end{equation}
for all $\phi, \psi \in (X, p)^*$ and $r \ge 0$.
The pair $((X, p)^*, p^*)$ is then a \textbf{normed cone} and every $x \in X$ induces a continuous functional on it with $\phi \mapsto x(\phi)$ that is linear in the sense of the usual compatibility with respect to linear combinations with non-negative scalars \cite[Proposition 2.4.22]{src:Cobzaş:FunctionalAnalysisinAsymmetricNormedSpaces}.
The above notation is consistent with the conventional normed setting, i.e. $(X, p)^*$ is just the ordinary continuous dual space whenever $p$ is a norm on $X$ and then $p^*$ is the conventional dual norm.
In either case $p^*$ similarly induces a first-countable topology on $(X, p)^*$.
Furthermore, we consider a \textbf{weak topology} on $(X, p)$, where a net $(x_\alpha)_{\alpha \in I}$ in $X$ converges $p$-weakly to $x \in X$, if $\limsup_\alpha \phi(x_\alpha) \le \phi(x)$ for all $\phi \in (X, p)^*$.
Again, if $p$ is a norm, the weak topology is the usual weak vector topology.
\begin{lemma}
\label{lem:WeakConvergence}
Let $(x_n)_{n \in \mathbb{N}}$ be a sequence in a vector space $X$ with topology $\tau$, $x \in X$ and $\phi \in (X, \tau)^*$ such that for every subsequence $(y_n)_{n \in \mathbb{N}}$,
\begin{equation}
\limsup_{n \to \infty} \phi \left( \frac{1}{n} \sum_{m = 1}^n y_n \right)
\le
\phi \left( x \right) \, .
\end{equation}
Then $\limsup_{n \to \infty} \phi( x_n ) \le \phi( x )$.
\end{lemma}
\begin{proof}
Suppose $\limsup_{n \to \infty} \phi( x_n ) \le \phi( x )$ does not hold.
Then there is some $\epsilon > 0$ and a subsequence $(y_n)_{n \in \mathbb{N}}$ such that $\phi( y_n ) \ge \phi(x) + \epsilon$ for all $n \in \mathbb{N}$.
This clearly contradicts the assumption.
\end{proof}
Given a function $f : (X, \tau) \to \overline{\mathbb{R}}$, define its \textbf{convex conjugate} $f^* : (X, \bar{\tau})^* \to \overline{\mathbb{R}}$ with $\phi \mapsto \sup_{x \in X} [ \phi(x) - f ( x ) ]$.
Moreover as in \cite{src:Zalinescu:ConvexAnalysisInGeneralVectorSpaces}, $\mathrm{dom}\, f = \{ x \in X : f(x) < \infty \}$ and $f$ is \textbf{proper} if $\mathrm{dom}\, f \neq \emptyset$ and $f$ does not attain the value $-\infty$.
\begin{definition}[{\cite{src:FischerZiebell:Tychonov}}]
Let $f : X \to \overline{\mathbb{R}}$ be a function on a normed cone $(X, p)$.
Then $f$ is \textbf{right Gâteaux differentiable} at a point $x \in X$ if $f(x) \in \mathbb{R}$ and there exists some continuous and linear $L : X \to \mathbb{R}$ such that
\begin{equation}
\lim_{ t  \searrow 0 } \frac{1}{t} \left[ f \left( x + t y \right) - f \left( x \right) - t L \left( y \right) \right] = 0
\end{equation}
for all $y \in X$.
In that case $L$ is called the \textbf{right Gâteaux derivative} of $f$ at $x$ and is written as $L = Df(x)$.
\end{definition}
\begin{theorem}[{\cite[Theorem 4.1]{src:FischerZiebell:Tychonov}}]
\label{thm:DualDifferentiability}
Let $(X, p)$ be an asymmetrically normed space and $f : X \to \overline{\mathbb{R}}$ a proper convex function, $x \in X$ and $\phi \in (X, \bar{p})^*$.
Setting $g = f - \phi$, the following statements are equivalent:
\begin{enumerate}
\item $f$ is $p$-weakly lower semicontinuous at $x$ and $f^*$ is right Gâteaux differentiable at $\phi$ with $D f^* (\phi) = x$.
Moreover, there exists a sequence $(y_n)_{n \in \mathbb{N}}$ such that
\begin{equation}
\lim_{n \to \infty} g(y_n) = \inf g(X) \qquad\text{and $y_n$ converges $p$-weakly to $x$.}
\end{equation}
\item $g(x) =  \inf g(X)$ and for every sequence $(y_n)_{n \in \mathbb{N}}$ in $X$ we have
\begin{equation}
\lim_{n \to \infty} g(y_n) = g(x) \implies \text{ $y_n$ converges $\bar{p}$-weakly to $x$.}
\end{equation}
\end{enumerate}
\end{theorem}
Given a Hausdorff locally convex space $X$, we consider the smallest $\sigma$-algebra $\mathcal{E}(X)$ for which all $\phi \in X^*$ are measurable and the Borel $\sigma$-algebra $\mathcal{B}(X)$ generated by all open sets in $X$.
A \textbf{measure} is always non-negative and $\sigma$-additive and a finite measure on $\mathcal{B}(X)$ is \textbf{Radon} if it is inner regular with respect to compact sets.
For any measure $\mu$ and any $\mu$-measurable function $f$, $[f]_\mu$ will denote the equivalence class of all functions being $\mu$-almost everywhere equal to $f$.
\begin{definition}[{\cite[Definition 2, p. 61]{src:Chevet:KernelOfMeasure}}]
Given a quasicomplete, Hausdorff, locally convex space $X$ and a Radon probability measure $\mu$ on $\mathcal{B}(X)$, an affine subspace $Y \subseteq X$ is a $\mu$-\textbf{Lusin affine subspace} of $X$ if for each $\epsilon > 0$, there exists a convex compact subset $K \subseteq Y$ such that $\mu(K) \ge 1 - \epsilon$.
Moreover, the $\mu$-\textbf{Lusin affine kernel} $\mathcal{A}(\mu)$ of $\mu$ is the intersection of all $\mu$-Lusin affine subspaces of $X$.
The intersection of all $\mu$-\textbf{Lusin linear subspaces} (i.e. those $\mu$-Lusin affine subspaces containing the origin) is denoted by $\mathcal{V}(\mu)$.
\end{definition}
A finite measure $\mu$ on $\mathcal{E}(X)$ is \textbf{scalarly of first order} whenever $X^* \subseteq L^1(\mu)$ in which case we define
\begin{equation}
\mathcal{K}(\mu) = \left\{ \left[ \phi - \int \phi \,\mathrm{d} \mu \right]_\mu \middle| \phi \in X^* \right\}
\end{equation}
and equip it with the $L^1(\mu)$ norm.
The completion of $\mathcal{K}(\mu)$ in the topology of convergence in $\mu$-measure is denoted by $\mathcal{M}(\mu)$ and we set $\mathcal{L}(\mu) = \mathcal{M}(\mu) \cap L^1(\mu)$.
On $\mathcal{L}(\mu)$, we consider the asymmetric norm $p_\mu : \mathcal{L}(\mu) \to [0, \infty)$ with $T \mapsto \int_X \max\{ -T, 0 \} \mathrm{d} \mu$ and equip $\mathcal{L}(\mu)$ with the topology $\tau_\mu$ given as the least upper bound of the topology of convergence in $\mu$-measure and the asymmetrically normed topology induced by $p_\mu$.
\begin{lemma}
$\tau_\mu$ is first-countable and its conjugate topology $\overline{\tau_\mu}$ is the least upper bound of the topology of convergence in $\mu$-measure and the asymmetrically normed topology induced by $\overline{p_\mu}$.
\end{lemma}
\begin{proof}
It is well-known that the topology of convergence in $\mu$-measure is metrisable and thus first-countable.
Being the least upper bound of two first-countable topologies, $\tau_\mu$ is thus first-countable as well.
Furthermore, a net $(T_\alpha)_{\alpha \in I}$ in $\mathcal{L}(\mu)$ $\overline{\tau_\mu}$-converges to some $T \in \mathcal{L}(\mu)$ precisely, when for each $\overline{\tau_\mu}$-neighbourhood $U$ of $T$, there is some $\alpha_U \in I$ be such that $T_\beta \in U$ for all $\beta \in I_{\ge \alpha_U}$.
By definition, this is equivalent to $-T_\alpha$ $\tau_\mu$-converging to $-T$, i.e.
\begin{equation}
- T_\alpha \to -T \text{ in $\mu$-measure and} \lim_\alpha p_\mu \left( - T_\alpha - \left( - T \right) \right) = \lim_\alpha p_\mu \left( T - T_\alpha \right) = 0 \, .
\end{equation}
This is in turn equivalent to $T_\alpha$ converging to $T$ in $\mu$-measure and $\lim_\alpha \overline{p_\mu} ( T_\alpha - T ) = 0$.
\end{proof}
A related construction to $\mathcal{A}(\mu)$ is the \textbf{kernel} $\mathcal{H}(\mu)$ of a finite measure $\mu$ on $\mathcal{E}(X)$ given as the continuous dual space of $\{ [ \phi ]_\mu : \phi \in X^* \}$ equipped with the topology of convergence in $\mu$-measure.
Note that if $\int_X \phi \mathrm{d} \mu = 0$ for all $\phi \in X^*$, then $\mathcal{H}(\mu) = \mathcal{M}(\mu)^*$.
Furthermore, it is known that $\mathcal{V}(\mu) = \mathcal{H}(\mu)$ for Radon probability measures on quasicomplete locally convex spaces \cite[p. 62, Theorem 3]{src:Chevet:KernelOfMeasure}.
For any measure $\mu$ on $\mathcal{E}(X)$ (or $\mathcal{B}(X)$) and any $x \in X$, the mapping $y \mapsto x + y$ is $\mu$-measurable such that we may consider the translated measure $\mu_x = \mu \ast \delta_x$ on $\mathcal{E}(X)$ (or $\mathcal{B}(X)$) where $\ast$ denotes convolution and $\delta_x$ is the Dirac measure at $x$.
\begin{theorem}[{\cite[Theorem 3, p. 62]{src:Chevet:KernelOfMeasure}}]
\label{thm:AffineKernelContainedInShiftedKernel}
Let $X$ be a quasicomplete, Hausdorff, locally convex space $X$, $\mu$ a Radon probability measure on $\mathcal{B}(X)$ and $x \in X$.
Then
\begin{equation}
\mathcal{A}(\mu)
\subseteq
x + \mathcal{H}(\mu_{-x}) \, .
\end{equation}
\end{theorem}
As in \cite{src:Chevet:KernelOfMeasure}, given a Hausdorff, locally convex space $X$, we define the algebraic dual space $\tilde{X}$ of $X^*$ and equip it with the weak-$\ast$ topology induced by $X^*$, turning it into a complete, Hausdorff, locally convex space with $(\tilde{X})^* = X^*$.
Then a finite measure $\mu$ on $\mathcal{E}(X)$ that is scalarly of first order has a \textbf{mean} $m_\mu \in \tilde{X}$ given by $\phi \mapsto \int_X \phi \mathrm{d} \mu$.
It is clear that the canonical inclusion map $X \to \tilde{X}$ is weakly continuous such that the pushforward measure $\tilde{\mu}$ is a measure on $\mathcal{E}(\tilde{X})$.
Furthermore, if $\mu$ is finite and Radon on $\mathcal{B}(X)$, then $\tilde{\mu}$ is finite and Radon on $\mathcal{B}(\tilde{X})$ and $X$ is $\tilde{\mu}$-measurable.
Consequently, if $X$ is quasicomplete, it is clear that $\mathcal{A}(\tilde{\mu})$ is equal to $\mathcal{A}(\mu)$ when $X$ is considered as a subset of $\tilde{X}$.

Finally, recall that a topological space $A$ is \textbf{hereditarily separable} if every subset of $A$ is separable.
\section{Positive Semidefinite \texorpdfstring{$\mu$}{μ}-measurable bilinear functionals}
\label{sec:MeasurableInnerProducts}
In this section, we introduce a bilinear analogue of the \emph{linear Lusin measurable functionals} considered in \cite{src:Slowikowski:PreSupports}.
However, while these objects certainly promise interesting (at least in the impression of the author) mathematics to uncover, in this paper we shall develop just what is needed for Wetterich's equation.
Consequently, the author does not want to use the label \enquote{bilinear Lusin measurable} and uses the nondescript term \enquote{regular} instead.

Throughout this section, $X$ is a Hausdorff, quasicomplete, locally convex space and $\mu$ a Radon probability measure on $X$.
Let $\mu^2$ denote the unique Radon extension of $\mu \times \mu$ to $\mathcal{B}(X \times X)$.
A particularly useful class of $\mu$-Lusin linear subspace $Y \subseteq X$ are the \textbf{standard} ones\cite[p. 56]{src:Chevet:KernelOfMeasure}, i.e. those for which there exists a sequence $(K_n)_{n \in \mathbb{N}}$ of compact disks in $X$ with $Y = \cup_{n \in \mathbb{N}} K_n$.
It is easy to show that every $\mu$-Lusin linear subspace contains a standard one (see \cite[p. 61]{src:Chevet:KernelOfMeasure}).
\begin{definition}
\label{def:MuRegularInnerProduct}
Let $f : X \times X \to \overline{\mathbb{R}}$ a $\mu^2$-measurable function for which there exists a sequence of continuous, symmetric, positive semidefinite, bilinear functionals $B_n : X \times X \to \mathbb{R}$ such that
\begin{itemize}
\item $m \le n$ implies $B_m(x,x) \le B_n(x,x)$ for every $x \in X$,
\item $\sup_{n \in \mathbb{N}} B_n(x,x) < \infty$ for $\mu$-almost every $x \in X$,
\item $B_n$ converges to $f$ in $\mu^2$-measure.
\end{itemize}
Then, we call $f$ a \textbf{$\mu$-regular semidefinite inner product} on $X$ and denote the set of all such objects by $\mathcal{I}(\mu)$.
\end{definition}
In fact, the first two points already define the $\mu$-regular semidefinite inner product uniquely.
\begin{theorem}
\label{thm:InnerProductRepresentation}
Let $B_n : X \times X \to \mathbb{R}$ be a sequence of continuous, symmetric, positive semidefinite, bilinear functionals with
\begin{itemize}
\item $m \le n$ implies $B_m(x,x) \le B_n(x,x)$ for every $x \in X$,
\item $\sup_{n \in \mathbb{N}} B_n(x,x) < \infty$ for $\mu$-almost every $x \in X$.
\end{itemize}
Then, $B_n$ converges to some $f \in \mathcal{I}(\mu)$ in $\mu^2$-measure.
Furthermore, there is a standard $\mu$-Lusin linear subspace $Z \subseteq X$ and a symmetric, positive semidefinite bilinear functional $B : Z \times Z \to \mathbb{R}$ with $B = f$ $\mu^2$-almost everywhere such that $B_n \upharpoonright Z \times Z$ converges pointwise to $B$.
Moreover, the convergence is uniform on $K \times K$ for every weakly compact, convex set $K \subseteq X$ contained in $Z$.
\end{theorem}
\begin{proof}
Consider the sequence $p_n(x) = \sqrt{B_n(x,x)}$ of continuous seminorms on $X$ and let $Y \subseteq X$ denote the set of all $x \in X$ with $\sup_{n \in \mathbb{N}} p_n(x) < \infty$.
Then $\mu(Y) = 1$ by assumption and from the homogeneity of $p_n$ and the triangle inequality, it is clear that $Y$ is a linear subspace.
By the monotonicity, $Y$ also coincides with the set of convergence points of $p_n$.
Now, define the symmetric, positive semidefinite, bilinear functional $B : Y \times Y \to \mathbb{R}$, with
\begin{equation}
B \left( x, y \right)
=
\lim_{n \to \infty} \frac{1}{4} \left[
B_n \left( x + y, x + y \right)
-
B_n \left( x - y, x - y \right)
\right] \, ,
\end{equation}
as well as the corresponding seminorm $p(x) = \sqrt{B(x,x)}$ on $Y$.
Evidently, $B_n \upharpoonright Y \times Y$ converges pointwise to $B$ and thus in $\mu^2$-measure, turning $f = [B]_{\mu^2}$ into a $\mu$-regular semidefinite inner product $X$.
Also, note that $B - B_n$ is a symmetric, positive semidefinite bilinear functional on $Y$ and thus, by Cauchy-Schwarz,
\begin{equation}
\left| B \left( x, y \right) - B_n \left( x, y \right) \right|
\le
\sqrt{B\left( x, x \right) - B_n \left( x, x \right)}
\sqrt{B\left( y, y \right) - B_n \left( y, y \right)} \, ,
\end{equation}
for all $x, y \in Y$.
Now, suppose $A \subseteq Y$ is such that $p_n^2$ converges uniformly to $p^2$ on $A$.
Clearly, we then also have uniform convergence on the balanced hull $B = A \cup (-A)$ of $A$.
Setting $\Delta_n = \sup_{x \in B} | p_n(x)^2 - p(x)^2 |$ we estimate,
\begin{equation}
\begin{aligned}
\sup_{\substack{x,y \in B\\t\in(0,1)}}
&\left| p_n \left( t x + \left[ 1 - t \right] y \right)^2 - p \left( t x + \left[ 1 - t \right] y \right)^2 \right| \\
&\le
\sup_{\substack{x,y \in B\\t\in(0,1)}}
\left( t^2 \Delta_n + \left[ 1 - t \right]^2 \Delta_n + 2 t \left( 1 - t \right) \left| B_n \left( x, y \right) - B \left( x, y \right) \right| \right) \\
&\le
\sup_{\substack{x,y \in B\\t\in(0,1)}}
\left( t^2 \Delta_n + \left[ 1 - t \right]^2 \Delta_n + 2 t \left( 1 - t \right) \Delta_n \right)
=
\Delta_n \, .
\end{aligned}
\end{equation}
Consequently, we also have uniform convergence on the convex balanced hull $C = \mathrm{cobal}\, A$ of $A$.
Now, let $x$ be a point in the closure of $C$ and $(x_\alpha)_{\alpha \in I}$ a net in $C$ converging to $x$.
Furthermore, for an arbitrary $\delta > 0$, take $N \in \mathbb{N}$ such that $\sup_{y \in C} | p_n(y)^2 - p_m(y)^2 | < \delta$ for all $m, n \in \mathbb{N}_{\ge N}$.
Then, for all $m, n \in \mathbb{N}_{\ge N}$,
\begin{equation}
\left| p_n \left( x \right)^2 - p_m \left( x \right)^2 \right|
=
\lim_{\alpha} \left| p_n \left( x_\alpha \right)^2 - p_m \left( x_\alpha \right)^2 \right|
<
\delta \, .
\end{equation}
Hence, $p_n(x)$ is a Cauchy sequence, $x \in Y$ and we have uniform convergence on the closure $D = \overline{C}$ of $C$.
Consequently, the uniform convergence on a set $A \subseteq Y$ implies the uniform convergence on $\overline{\mathrm{cobal}\, A} \subseteq Y$.
Combined with Egoroff's theorem, for each $m \in \mathbb{N}$, we find a closed disk $D \subseteq Y$ with $\mu(Y \setminus D) < 1/(2m)$ on which $p_n^2$ converges uniformly to $p^2$.
Since $\mu$ is Radon, it follows that there is a compact disk $Q_m \subseteq D$ with $\mu(Q_m) > 1 - 1/m$.
From the bilinearity, we also have uniform convergence on $2 Q_m$ such that without loss of generality, $2 Q_m \subseteq Q_{m+1}$.
Consequently, $Z = \cup_{m \in \mathbb{N}} Q_m$ is a standard $\mu$-Lusin linear subspace.
Furthermore, by polarisation
\begin{equation}
\begin{aligned}
\sup_{x, y \in Q_m} &\left| B_n \left( x, y \right) - B \left( x, y \right) \right| \\
&=
\sup_{x, y \in Q_m} \left| p_n \left( \frac{x + y}{2} \right)^2 - p_n \left( \frac{x - y}{2} \right)^2 + p \left( \frac{x + y}{2} \right)^2 - p \left( \frac{x + y}{2} \right)^2 \right| \\
&\le
2 \sup_{x \in Q_m} \left| p_n \left( x \right)^2 - p \left( x \right)^2 \right| \, .
\end{aligned}
\end{equation}
Hence, $B_n$ converges uniformly to $B$ on $Q_m \times Q_m$.
Finally, let $K \subseteq X$ be convex, weakly compact and contained in $Z$.
Then the convex hull of $K \cup -K$ is a weakly compact disk (see \cite[Theorem 10.2]{src:SchaeferWolff:TVS}) and thus a Banach disk.
Hence, the reasoning in \cite[p. 56]{src:Chevet:KernelOfMeasure} applies.
\end{proof}
\begin{definition}
\label{def:StandardRepresentation}
Let $Y \subseteq X$ be a standard $\mu$-Lusin linear subspace and $B: Y \times Y \to \mathbb{R}$ a symmetric, positive semidefinite, bilinear functional that is continuous on every set of the form $K \times K$ where $K \subseteq X$ is a weakly compact convex set contained in $Y$.
If $f \in \mathcal{I}(\mu)$ with $f = [B]_{\mu^2}$, $(Y, B)$ is called a \textbf{standard representation} of $f$.
\end{definition}
By \cref{thm:InnerProductRepresentation}, every $f \in \mathcal{I}(\mu)$ has a standard representation.
\begin{remark}
\label{rem:LinearAlmostUniformContinuity}
Let $\phi : Y \to \mathbb{R}$ be a linear functional on a standard $\mu$-Lusin linear subspace $Y = \cup_{n \in \mathbb{N}} K_n \subseteq X$ where $K_n$ is a sequence of compact disks.
Then $\phi$ is continuous on each $K_n$ if and only if it is continuous on every convex, weakly compact subset of $X$ contained in $Y$. (see \cite[p. 56]{src:Chevet:KernelOfMeasure}).
\end{remark}
The continuity of such representations can be combined with a separability assumption to yield a strong uniqueness statement.
\begin{lemma}
\label{lem:EqualityOfInnerProductRepresentations}
Let $Y, Z \subseteq X$ be two linear subspaces with $A : Y \times Y \to \mathbb{R}$ and $B : Z \times Z \to \mathbb{R}$ two bilinear functionals with $A = B$ $\mu^2$-almost everywhere.
Suppose there are two sequences of hereditarily separable, compact disks $(K_n)_{n \in \mathbb{N}}$ and $(L_n)_{n \in \mathbb{N}}$ in $Y$ and $Z$ respectively, such that for all $n \in \mathbb{N}$,
\begin{itemize}
\item $\mu(K_n) > 1 - 1/n$ and $\mu(L_n) > 1 - 1/n$,
\item $A \upharpoonright K_n \times K_n$ and $B \upharpoonright L_n \times L_n$ are continuous.
\end{itemize}
Then, there is a $\mu$-Lusin linear subspace $U \subseteq Y \cap Z$ such that $A$ and $B$ coincide on $U \times U$.
\end{lemma}
\begin{proof}
Without loss of generality, $2 K_n \subseteq K_{n+1}$, $2 L_m \subseteq L_{n+1}$ and $Y = \cup_{n \in \mathbb{N}} K_n$ respectively $Z = \cup_{n \in \mathbb{N}} L_n$.
Let
\begin{equation}
S = \left\{ x \in Y \cap Z \middle| A \left( x, \cdot \right) = B \left( x, \cdot \right) \; \text{$\mu$-almost everywhere} \right\} \, .
\end{equation}
Since $A$ and $B$ coincide $\mu^2$-almost everywhere, it is clear that $S$ has full $\mu$-measure and, by the bilinearity of $A$ and $B$, $S$ is a linear space.
Moreover, by \cref{rem:LinearAlmostUniformContinuity}, for each $x \in S$, $A(x, \cdot)$ and $B(x, \cdot)$ are representations of the same linear Lusin $\mu$-measurable functional (see \cite[Section 3]{src:Slowikowski:PreSupports}).
Hence, by \cite[Proposition 3.1]{src:Slowikowski:PreSupports}, there is a standard $\mu$-Lusin linear subspace $V_x$ on which they coincide.
Now, for every $n \in \mathbb{N}$, pick a compact set $C_n \subseteq S$ with $\mu(C_n) > 1 - n/3$ and $2 C_n \subseteq C_{n+1}$ and set
\begin{equation}
Q_n = \overline{\mathrm{cobal}\, C_n \cap K_{3n} \cap L_{3n}}
\qquad\text{noting that}\qquad
\mu \left( Q_n \right) > 1 - \frac{1}{n} \, .
\end{equation}
Since $K_{3n} \cap L_{3n}$ is a compact disk, $Q_n \subseteq K_{3n} \cap L_{3n}$ and by assumption both $A$ and $B$ are continuous on $Q_n \times Q_n$.
Note that $W = \cup_{n \in \mathbb{N}} Q_n$ is a linear space and therefore a $\mu$-Lusin linear subspace of $X$.
For each $n \in \mathbb{N}$, fix a countable, dense subset $D_n$ of $\mathrm{cobal}\, C_n \cap K_{2n} \cap L_{2n}$ and note that $D_n \subseteq S$ since $A$ and $B$ are bilinear.
The union $D = \cup_{n \in \mathbb{N}} D_n$ is also countable and it is easy to see (e.g. \cite[Proposition 2.4]{src:Slowikowski:PreSupports}) that $U = \cap_{z \in D} V_z \cap W$ is a $\mu$-Lusin linear subspace.
Now, pick any $x \in W$ and fix some $m \in \mathbb{N}$ with $x \in Q_m$ and a net $(x_\alpha)_{\alpha \in I}$ in $D_m$ converging to $x$ and let $y \in U$ be arbitrary.
Then, there is some $M \in \mathbb{N}_{\ge m}$ such that $y \in Q_M$.
Since $Q_m \subseteq Q_M$, we use the continuity to find
\begin{equation}
A \left( x, y \right)
=
\lim_{\alpha} A \left( x_\alpha, y \right)
=
\lim_{\alpha} B \left( x_\alpha, y \right)
=
B \left( x, y \right) \, .
\end{equation}
Consequently, $A$ and $B$ coincide on $W \times U$ and thus on $U \times U$.
\end{proof}
\begin{remark}
There are simple sufficient conditions for the separability condition to be satisfied, such as the existence of a sequence $(K_n)_{n \in \mathbb{N}}$ of metrisable compact disks in $X$ with $\mu(\cup_{n \in \mathbb{N}} K_n) = 1$.
It is also automatically satisfied, if $X$ is a Souslin space (see e.g. \cite[Corollary 6.7.8]{src:Bogachev:MeasureTheory}).
\end{remark}
In the following, we shall assume the existence of a sequence $(K_n)_{n \in \mathbb{N}}$ of hereditarily separable, compact disks in $X$ with $\mu(\cup_{n \in \mathbb{N}} K_n) = 1$.
\begin{corollary}
\label{cor:InnerProductWelldefinedOnLusinSubspace}
Let $f \in \mathcal{I}(\mu)$ and $(Y, B)$ as well as $(Z, C)$ be two standard representations of $f$.
Then there is a standard $\mu$-Lusin linear subspace $W \subseteq Y \cap Z$ such that $B$ and $C$ coincide on $W \times W$ and $(W , B \upharpoonright W \times W)$ is a standard representation of $f$.
\end{corollary}
It follows that with the separability condition, every $f \in \mathcal{I}(\mu)$ has a $\mu$-almost everywhere well-defined restriction to the diagonal $\tilde{f} : X \to \overline{\mathbb{R}}, x \mapsto \lim_{n \to \infty} B(x,x)$ for any standard representation $(Y, B)$ of $f$.
This is reminiscent of the restriction of the integral kernel of a trace-class operator to the diagonal treated in \cite{src:Brislawn:TraceableIntegralKernels} and to the concept of virtual continuity considered in \cite{src:VershikEtAl:VirtualContinuity}.
Another consequence is that $f$ has an everywhere well-defined restriction $\underline{f} : \mathcal{H}(\mu) \times \mathcal{H}(\mu)$ (though this would also follow without the separability condition).
\begin{proposition}
\label{prop:TranslationOfInnerProduct}
Let $f \in \mathcal{I}(\mu)$, $h \in \mathcal{H}(\mu)$ and $(Y, B)$ a standard representation of $f$.
Then $Y$ is a Borel set, $B$ is Borel measurable and for every $h \in \mathcal{H}(\mu)$,
the function
\begin{equation}
B^h : Y \times Y \to \mathbb{R},
x \mapsto B \left( x - h, x - h \right)
\end{equation}
is $\mu$-measurable.
Moreover, $Y$ is a $\mu_{-h}$-Lusin linear subspace of $X$ and $(Y,B)$ is a $\mu_{-h}$-standard representation of $[B]_{\mu_{-h}^2} \in \mathcal{I}(\mu_{-h})$.
\end{proposition}
\begin{proof}
Since $Y = \cup_{n \in \mathbb{N}} L_n$ for some sequence $(L_n)_{n \in \mathbb{N}}$ of compact disks, it is Borel.
Furthermore, $B$ is the pointwise limit of the Borel measurable functions $B \chi_n$ where $\chi_n$ is the characteristic function of $K_n \times K_n$.
Hence, $B$ is Borel measurable.
That $B^h$ is Borel measurable follows because $\mathcal{H}(\mu) \subseteq Y$ and the fact that translations are Borel measurable.
Now, pick any sequence $(B_n)_{n \in \mathbb{N}}$ converging to $f$ in the sense of \cref{def:MuRegularInnerProduct}.
Then the limit as in \cref{thm:InnerProductRepresentation} coincides with $B$ on some standard $\mu$-Lusin linear subspace $W \subseteq Y$ by \cref{cor:InnerProductWelldefinedOnLusinSubspace} on which the convergence happens pointwise.
Since $h \in W$, we thus have that $W$ is a $\mu_{-h}$-Lusin linear subspace and that $(B_n)_{n \in \mathbb{N}}$ converges to $B \upharpoonright W \times W$ in $\mu_{-h}^2$-measure.
\end{proof}
In that sense, every $f \in \mathcal{I}(\mu)$ uniquely corresponds to a $g \in \mathcal{I}(\mu_{-h})$.
It is tempting to also define an associated \enquote{seminorm}.
But there is of course no way to uniquely define the corresponding open and closed \enquote{balls} around points $x \in X$.
However, by \cref{prop:TranslationOfInnerProduct},
\begin{equation}
K^A_\epsilon(h)
=
\left\{ x \in Y : \sqrt{A \left( x - h, x - h \right)} \le \epsilon \right\}
=
h + K^A_\epsilon(0)
\end{equation}
is $\mu$-measurable for every $h \in \mathcal{H}(\mu)$ and $\mu ( K^A_\epsilon(h) )$ is independent of the chosen standard representation $(Y, A)$.
Hence, we abuse the notation and define
\begin{equation}
\mu \left( K^f_\epsilon(h) \right)
:=
\mu \left( K^A_\epsilon(h) \right) \, .
\end{equation}
It is easily seen that the set $\mathcal{I}(\mu)$ of $\mu$-regular inner products is closed under finite sums.
Consequently, there is a natural partial order on $\mathcal{I}(\mu)$ with $f \le g$ for $f, g \in \mathcal{I}(\mu)$ if there is some $h \in \mathcal{I}(\mu)$ such that $f + h = g$.
With the preceeding results, this ordering is now easily understood.
\begin{theorem}
\label{thm:InnerProductOrderDiagonal}
Let $f, g \in \mathcal{I}(\mu)$ with $f \le g$.
Then $\tilde{f} \le \tilde{g}$ $\mu$-almost everywhere and $\underline{f}(x,x) \le \underline{g}(x,x)$ for all $x \in \mathcal{H}(\mu)$.
\end{theorem}
\begin{proof}
Let $h \in \mathcal{I}(\mu)$ with $f + h = g$.
Using \cref{thm:InnerProductRepresentation} as well as the fact that the family of $\mu$-Lusin linear subspaces is stable under finite intersections, there are standard representations $(Y, B^1), (Y, B^2)$ and $(Y, B^3)$ of $f, g$ and $h$ respectively.
By assumption, $B^1 + B^3 = B^2$ $\mu^2$-almost everywhere such that by \cref{lem:EqualityOfInnerProductRepresentations}, there is a $\mu$-Lusin linear subspace $Z \subseteq Y$ such that $B^1 + B^3$ and $B^2$ coincide on $Z \times Z$.
Consequently, $B^1(x,x) \le B^2(x,x)$ for all $x \in Z$.
\end{proof}
\begin{theorem}
\label{thm:BoundedIncreasingInnerProductsConverge}
Let $(f_n)_{n \in \mathbb{N}}$ be an increasing sequence in $\mathcal{I}(\mu)$ such that for $\mu$-almost every $x \in X$, $\sup_{n \in \mathbb{N}} \tilde{f}_n(x) < \infty$.
Then there is a unique $f \in \mathcal{I}(\mu)$ such that $f_n$ converges to $f$ in $\mu^2$-measure.
Moreover, for each $n \in \mathbb{N}$, letting $(Y_n,B_n)$ be a standard representation of $f_n$, $\tilde{f}(x) = \sup_{n \in \mathbb{N}} B_n(x,x)$ for $\mu$-almost every $x \in X$ and 
\begin{equation}
\label{eq:MeasureOfLimitBalls}
\mu \left( K^f_\epsilon (h) \right)
=
\inf_{n \in \mathbb{N}} \mu \left( K^{f_n}_\epsilon (h) \right)
=
\lim_{n \to \infty} \mu \left( K^{f_n}_\epsilon (h) \right)
\end{equation}
for all $\epsilon > 0$ and $h \in \mathcal{H}(\mu)$.
\end{theorem}
\begin{proof}
Fix some translation invariant metric $d : L^0(\mu^2) \times L^0(\mu^2) \to [0, \infty)$ inducing the topology of convergence in $\mu^2$-measure.
Set $Q_1 = \{ 0 \} \subseteq X$,
\begin{equation}
C_1 : X \times X \to \mathbb{R}, (x,y) \mapsto 0 \quad\text{and}\quad \eta = d(C_1, B_1) + 1 \, .
\end{equation}
Clearly, $C_1 \in \mathcal{I}(\mu)$, $C_1 \le B_1$ and $Q_1 \subseteq Y_1$ is a compact disk in $X$.
For the induction step, suppose that for some $N \in \mathbb{N}$, we have continuous, symmetric, positive semidefinite, bilinear functionals $(C_n)_{n = 1}^N$ on $X$ and a sequence $(Q_n)_{n = 1}^N$ of compact disks in $X$ such that
\begin{itemize}
\item $\forall m, n \in \{1, \dots, N\} \forall x \in X$, $m \le n \implies C_m ( x, x ) \le C_n ( x, x )$,
\item $C_N \le B_N$,
\item $d(C_N, B_N) < \eta / N$,
\item $Q_N \subseteq Y_N$ with $\mu(Q_N) \ge 1 - 2^{-N}$,
\item $\sup_{x \in Q_N} \vert B_N (x, x) - C_N(x,x) \vert < 1/N$.
\end{itemize}
Then $C_N \le f_{N+1}$ by assumption, such that there is a $g_N \in \mathcal{I}(\mu)$ with $C_N + g_N = f_{N+1}$.
Picking a standard representation $(Z_N, D_N)$ of $g_N$ and applying \cref{def:MuRegularInnerProduct,thm:InnerProductRepresentation}, there is a continuous, symmetric, positive semidefinite, bilinear functional $E_N$ on $X$ such that $E_N \le g_N$ and $d(E_N,D_N) < \eta/(N+1)$ as well as a compact disk $Q_{N+1} \subseteq Z_N \cap Y_{N+1}$ with $\mu(Q_{N+1}) \ge 1 - 2^{-N-1}$ such that
\begin{equation}
\sup_{x \in Q_{N+1}} \left\vert D_N \left( x \right) - E_N \left( x, x \right) \right\vert
<
\frac{1}{N+1} \, .
\end{equation}
Setting $C_{N+1} = C_N + E_N$, we clearly have $C_{N+1} \le f_{N+1}$ and
\begin{equation}
d \left( C_{N+1}, B_{N+1} \right)
=
d \left( C_{N+1} - C_N, B_{N+1} - C_N \right)
=
d \left( E_N, D_N \right)
<
\frac{\eta}{N+1} \, .
\end{equation}
Moreover,
\begin{equation}
\sup_{x \in Q_{N+1}} \left\vert C_{N+1} \left( x, x \right) - B_{N+1} \left( x \right) \right\vert
=
\sup_{x \in Q_{N+1}} \left\vert E_N \left( x, x \right) - D_N \left( x \right) \right\vert
<
\frac{1}{N+1} \, .
\end{equation}
Consequently, we have a sequence $(C_n)_{n \in \mathbb{N}}$ such that the above points hold for all $N \in \mathbb{N}$.
It is clear that $\sup_{n \in \mathbb{N}} C_n(x,x) < \infty$ for $\mu$-almost every $x \in X$ such that by \cref{thm:InnerProductRepresentation}, $C_n$ converges to some $f \in \mathcal{I}(\mu)$.
Moreover,
\begin{equation}
d \left( f_n, f \right)
\le
d \left( f_n, C_n \right)
+
d \left( C_n, f \right)
<
\frac{\eta}{n}
+
d \left( C_n, f \right)
\to 0
\left( n \to \infty \right) \, ,
\end{equation}
such that $f_n$ converges to $f$ in $\mu^2$-measure.
The uniqueness is clear, since $(L^0(\mu^2), d)$ is Hausdorff.
To see the formula for $\tilde{f}$, note that for every $N \in \mathbb{N}$ and $\mu$-almost every $x \in \cap_{n = N}^\infty Q_n$, 
\begin{equation}
\label{eq:fTildeLimit}
\tilde{f} \left( x \right)
=
\sup_{n \in \mathbb{N}} C_n \left( x, x \right)
\le
\sup_{n \in \mathbb{N}} B_n \left( x, x \right)
\le
\limsup_{n \in \mathbb{N}} \left[ C_n \left( x, x \right) + \frac{1}{n} \right]
=
\sup_{n \in \mathbb{N}} C_n \left( x, x \right) \, .
\end{equation}
Now the claim follows, because
\begin{equation}
\mu \left( \bigcap_{n = N}^\infty Q_n \right)
\ge
1 - 2^{1-N}
\xrightarrow{N \to \infty} 1\, .
\end{equation}
For \cref{eq:MeasureOfLimitBalls}, let $h \in \mathcal{H}(\mu)$ and $\epsilon > 0$.
By \cref{thm:InnerProductRepresentation}, there is a standard representation $(Y, B)$ of $f$ such that for all $x \in Y$, $C_n(x,x)$ and thus also $B_n(x,x)$ converges pointwise to $B(x,x)$.
Then, for each $n \in \mathbb{N}$, $Y \cap K^{B_{n+1}}_\epsilon (h) \subseteq Y \cap K^{B_n}_\epsilon (h)$ and
\begin{equation}
\mu \left( K^{f_n}_\epsilon (h) \right)
=
\mu \left( Y \cap K^{B_n}_\epsilon (h) \right) \, .
\end{equation}
Furthermore, \cref{eq:fTildeLimit} shows that $K^B(h) = \cap_{n \in \mathbb{N}} Y \cap K^{B_n}_\epsilon (h)$.
\end{proof}
Though these results suffice for this paper, there are certainly further interesting questions about these inner products:
\begin{itemize}
\item Is positive semidefiniteness necessary or can one treat general symmetric bilinear forms?
\item Is every pair $(Y, B)$ as in \cref{def:StandardRepresentation} a standard representation of some $f \in \mathcal{I}(\mu)$?
\item Is the separability condition necessary?
\item Is every (positive semidefinite) symmetric bilinear form on a $\mu$-Lusin linear subspace that is continuous on convex, weakly compact subsets the almost uniform limit of continuous symmetric bilinear forms?
\item Is there a converse to \cref{thm:InnerProductOrderDiagonal}?
\end{itemize}
\section{Wetterich's Equation}
Throughout this section we fix a quasicomplete, Hausdorff, locally convex space $X$, some point $w \in X$ and a Radon probability measure $\mu$ on $\mathcal{B}(X)$ for which $\exp[\phi] \in L^1(\mu)$ for all $\phi \in X^*$.
Futhermore, we shall make the assumption that there is a sequence $(K_n)_{n \in \mathbb{N}}$ of hereditarily separable, compact disks in $X$ with $\mu(\cup_{n \in \mathbb{N}} K_n) = 1$.
Again, if $X$ is a Souslin space, this assumption is automatically satisfied (see e.g. \cite[Corollary 6.7.8]{src:Bogachev:MeasureTheory}).
Denote by $\Xi$ the $\sigma$-algebra on $(X^*)_{\mathbb{C}}$ generated by the dual pair $(X_{\mathbb{C}}, (X^*)_{\mathbb{C}})$, i.e. $\Xi = \mathcal{E}((X^*)_{\mathbb{C}})$ when $(X^*)_{\mathbb{C}}$ is considered with its weak-$\ast$ topology.
\subsection{The Derivation}
Fix an increasing family $(I_k)_{k \ge 0}$ of $\mu_{-w}$-regular inner products on $X$, set $Q_k( x ) := \tilde{I}_k( x - w )$ for brevity and define the following probability measures on $\mathcal{B}(X)$:
\begin{equation}
\mu_k
=
\frac{\exp \left[ - \frac{1}{2} Q_k \right] \mu}{\left\Vert \exp \left[ - \frac{1}{2} Q_k \right] \right\Vert_{L^1(\mu)}}
=:
\frac{1}{N_k} \exp \left[ - \frac{1}{2} Q_k \right] \mu \, .
\end{equation}
Now, we demand that $(I_k)_{k \ge 0}$ is chosen such that the following hold:
\begin{enumerate}
\item For every $k \ge 0$, there is a $\sigma$-finite measure space $(A_k, \Sigma_k, \kappa_k)$ and a $(\Sigma_k, \Xi)$-measurable mapping $\omega^k : A_k \to (X^*)_\mathbb{C}, a \mapsto \omega^k_a$ such that for every $y \in \mathcal{H}(\mu_{-w})$, the function $a \mapsto \omega^k_a(y)$ is in $L^2_{\mathbb{C}}(\kappa_k)$ and
\begin{equation}
\label{eq:QlPrimeOnKernelOfMuMinusW}
D^+ \underline{I_k} \left( y, y \right) 
=
\int_{A_k} \left| \omega^k_a \left( y \right) \right|^2 \mathrm{d} \kappa_k \left( a \right)
 \, ,
\end{equation}
where the Dini derivative is taken with respect to $k$.
\end{enumerate}
For brevity, define $Q_k' : X \to \overline{\mathbb{R}}$ with
\begin{equation}
x \mapsto \int_{A_k} \left| \omega^k_a \left( x - w \right) \right|^2 \mathrm{d} \kappa_k \left( a \right) \, .
\end{equation}
\begin{enumerate}
\setcounter{enumi}{1}
\item For every $l \ge 0$, $Q_l'$ is $\mu$-measurable.
Furthermore, for every sequence $(k_n)_{n \in \mathbb{N}}$ in $[0, \infty) \setminus \{ l \}$ with $l = \lim_{n \to \infty} k_n$, there is a sequence $(r_n)_{n \in \mathbb{N}}$ in $(0, \infty)$ with $\lim_{n \to \infty} r_n = \infty$ and
\begin{equation}
\label{eq:QDerivativeCompatibility}
\lim_{n \to \infty} \int_X \exp \left[ \left| \frac{r_n \left( Q_l - Q_{k_n} \right)}{2 \left( l - k_n \right) } - \frac{r_n}{2} Q_l' \right| - \frac{Q_l}{2} \right] \mathrm{d} \mu
=
\int_X \exp \left[ - \frac{Q_l}{2} \right] \mathrm{d} \mu \, .
\end{equation}
\item For every $k \ge 0$, there is some $R > 0$ such that
\begin{equation}
\label{eq:QDerivativeBound}
\int_X \exp \left[ \frac{1}{2 R} Q_k' - \frac{1}{2} Q_k \right] \mathrm{d} \mu < \infty \, .
\end{equation}
\end{enumerate}
\begin{remark}
First, note that in the usual framework of Wetterich's equation, $w = 0$ (see e.g \cite{src:Gies:IntroductionToFRG}).
However, admitting other $w$ is meaningful in order to understand the corresponding boundary conditions as shown in \cref{sec:BoundaryConditions}.
Also, while the properties that $(I_k)_{k \ge 0}$ is demanded to have may look very technical at first glance, they are tailored towards being as general as possible to allow the proofs to work.
If possible, one could of course just pick continuous, symmetric and positive semidefinite representatives of $I_k$ eliminating the need for measure-theoretic arguments.
But if e.g. it is known that $\mu$ is supported on a smaller linear subspace with a natural stronger topology, it can be beneficial to use $\mu_{-w}$-regular inner products instead of continuous ones, because one then can keep $X^*$ and need not enlargen it to encompass the stronger topology.
It is of no surprise, that a certain compatibility between the family $(Q_k)_{k \ge 0}$ and the measure $\mu$ is necessary.
\cref{eq:QDerivativeCompatibility} is a rather strong differentiability requirement which is needed to show (rather crude) estimates with the Orlicz-Hölder inequality.
Similarly, \cref{eq:QDerivativeBound} adds strong restrictions on the derivative $Q_k'$.
Incidentally, by Fernique's theorem, \cref{eq:QDerivativeBound} is automatically satisfied if $\mu$ is absolutely continuous with respect to a Radon Gaußian mesaure $\nu$ with a density in $L^q(\nu)$ for some $q > 1$.
\end{remark}
Before actually defining the objects we are interested in, let us collect some important immediate consequences.
\begin{corollary}
\label{cor:QlPrimeAsLimitOfDifferenceQuotients}
For each $l \ge 0$ and every sequence $(k_n)_{n \in \mathbb{N}}$ in $[0, \infty) \setminus \{l\}$ converging to $l$, the functions $Q_{k_n}$ converge to $Q_l$ in $\mu$-measure.
Furthermore, $(Q_l - Q_{k_n})/(l - k_n)$ converges to $Q_l'$ in $\mu$-measure.
\end{corollary}
\begin{lemma}
\label{lem:NkRightDiffability}
The function $[0, \infty) \to (0, \infty), k \mapsto N_k$ is right differentiable with derivative,
\begin{equation}
\label{eq:NkDerivative}
N_k'
=
- \frac{1}{2} \int_X Q_k' \exp \left[ - \frac{1}{2} Q_k \right] \mathrm{d} \mu \, .
\end{equation}
\end{lemma}
\begin{proof}
First, note that the right-hand side of \cref{eq:NkDerivative} is finite by \cref{eq:QDerivativeBound}.
Now, let $l \ge 0$, $(k_n)_{n \in \mathbb{N}}$ be a strictly monotonically decreasing sequence in $[0, \infty)$ converging to $l$ and set $\nu = \exp[-Q_l / 2] \mu$.
Note that the function $x \mapsto e^x$ on $(-\infty, 0]$ has Lipschitz constant one and $Q_l \le Q_{k_n}$ $\mu$-almost everywhere by \cref{thm:InnerProductOrderDiagonal}.
Hence,
\begin{equation}
\label{eq:ExponentialQDifferenceDecreasingEstimate}
\begin{aligned}
&\left| \frac{1}{l - k_n} \left( \exp \left[ - \frac{Q_l}{2} \right] - \exp \left[ - \frac{Q_{k_n}}{2} \right] \right) \right| \\
&=
\frac{1}{k_n - l} \left( 1 - \exp \left[ \frac{Q_l - Q_{k_n}}{2} \right] \right) \exp \left[ - \frac{Q_l}{2} \right]
\le
\frac{Q_l - Q_{k_n}}{2 \left( l - k_n \right)} \exp \left[ - \frac{Q_l}{2} \right] \, ,
\end{aligned}
\end{equation}
By using the estimate $|x| \le e^{|x|} - 1$ and \cref{eq:QDerivativeCompatibility}, it is easy to see that
\begin{equation}
\lim_{n \to \infty} \int_X \left| \frac{Q_l - Q_{k_n}}{2 \left( l - k_n \right)} - Q_l' \right| \exp \left[ - \frac{Q_l}{2} \right] \mathrm{d} \mu
=
0 \, .
\end{equation}
Hence, $k \mapsto N_k$ is right-differentiable with the derivative given in \cref{eq:NkDerivative}.
\end{proof}
\begin{corollary}
\label{cor:LogNkAbsoluteContinuity}
For every $k > 0$, the function $a \mapsto \ln N_a$ is absolutely continuous.
\end{corollary}
\begin{proof}
Using the dominated convergence theorem and \cref{cor:QlPrimeAsLimitOfDifferenceQuotients}, the function $a \mapsto - \ln N_a$ continuous.
Since it is also monotonically increasing and right differentiable, \cref{lem:FiniteDiniImpliesAbsoluteContinuity} applies.
\end{proof}
With regards to the constructions in \cref{sec:Preliminaries}, in the following we shall for brevity write $k$ as an index instead of $\mu_k$.
The goal of this paper is to study the convex conjugates of the functions $V_k : X^* \to \mathbb{R}$ given by $\phi \mapsto \ln \int_X \exp[\phi] \mathrm{d} \mu_k$ and the first observation is that the domain of the convex conjugate $V_k^*$ should be some subset of $\tilde{X}$.
Before restricting ourselves to a particularly useful subset, note that we immediately obtain a continuity property using the monotonicity of $k \mapsto Q_k$ and a $\Gamma$-convergence-like argument.
\begin{theorem}
\label{thm:VAstContinuity}
Let $y \in \tilde{X}$ and define $V_k^*(y) = \sup_{\phi \in X^*} [ y(\phi) - V_k ( \phi ) ]$.
Then, for every $l > 0$ with $V_l^*(y) < \infty$, the function $V^*(y) : [0, l] \to \mathbb{R}, k \mapsto V_k^*(y)$ is continuous.
\end{theorem}
\begin{proof}
Note that for all $0 \le a \le b \le l$ and all $\phi \in X^*$,
\begin{equation}
V_a \left( \phi \right)
=
\ln \int_X \exp \left[ \phi \right] \mathrm{d} \mu_a
\ge
\ln \int_X \exp \left[ \phi \right] \mathrm{d} \mu_b
+
\ln \frac{N_b}{N_a}
=
V_b \left( \phi \right)
+
\ln \frac{N_b}{N_a} \, .
\end{equation}
Hence,
\begin{equation}
\label{eq:VAstUppersemicontinuityContinuityFromTheLeft}
V_a^* \left( y \right)
\le
V_b^* \left( y \right)
-
\ln \frac{N_b}{N_a}
\quad\text{and using \cref{cor:LogNkAbsoluteContinuity},}\quad
\limsup_{a \nearrow b} V_a^* \left( y \right) \le V_b^* \left( y \right) \, .
\end{equation}
Now, let $(\phi_n)_{n \in \mathbb{N}}$ be a sequence in $X^*$ with $\lim_{n \to \infty} [ y ( \phi_n ) - V_b( \phi_n ) ] = V_b^*(y)$ and let $(a_n)_{n \in \mathbb{N}}$ be a monotonically increasing sequence in $[0, b]$ with $\sup_{n \in \mathbb{N}} a_n = b$.
Recalling that $Q_{a_n} \le Q_b$ $\mu$-almost everywhere and that $Q_{a_n}$ converges to $Q_b$ in $\mu$-measure, we have $\lim_{m \to \infty} V_{a_m}( \phi_n ) = V_b( \phi_n )$ for all $n \in \mathbb{N}$, by the dominated convergence theorem using the fact that $\exp[\phi_n] \in L^1(\mu)$.
Hence, for every $n \in \mathbb{N}$, we find some $N \in \mathbb{N}$ such that for all $m \in \mathbb{N}_{\ge N}$,
\begin{equation}
V_{a_m} \left( \phi_n \right) \le V_b \left( \phi_n \right) + \frac{1}{n} \, .
\end{equation}
Consequently,
\begin{equation}
\liminf_{m \to \infty} V_{a_m}^* \left( y \right)
\ge
\liminf_{m \to \infty} \left[
y \left( \phi_n \right) - V_{a_m} \left( \phi_n \right)
\right]
\ge
y \left( \phi_n \right) - V_b \left( \phi_n \right) - \frac{1}{n} \, .
\end{equation}
Thus, $\liminf_{m \to \infty} V_{a_m}^* ( y ) \ge V_b^*( y )$ and with \cref{eq:VAstUppersemicontinuityContinuityFromTheLeft}, $V^*(y)$ is left-continuous at $b$.
The right-continuity at $a$ follows analogously and the claim follows.
\end{proof}
However, it is à priori much less clear on which subset of $\tilde{X}$ one should expect Wetterich's equation to hold.
In \cite{src:Ziebell:RigorousFRG}, the author considered measures $\mu$ that are absolutely continuous with respect to a suitable centred Radon Gaußian measure $G$ and the domain of the effective average action turned out to be the Cameron-Martin space of $G$.
Guided by that result, we define the domain of $V_k^*$ to be $\mathcal{A}(\mu_k) = \mathcal{A}(\mu)$ for all $k \ge 0$.
To see the benefit of this restriction, let $\nu_k = (\widetilde{\mu_k})_{- m_k}$ denote the translated version of $\mu_k$ on $\mathcal{B}(\tilde{X})$ with zero mean.
Then it is clear from \cref{thm:AffineKernelContainedInShiftedKernel} that 
\begin{equation}
\label{eq:AffineKernelInMeanShift}
\mathcal{A}(\mu) \subseteq m_k + \mathcal{H}( \nu_k )
\end{equation}
for all $k \ge 0$.
Hence, for every $y \in \mathcal{A}(\mu)$ and all $k, l \ge 0$, there is some $z \in \mathcal{H}( \nu_l )$ such that
\begin{equation}
y \left( \phi \right) - V_k \left( \phi \right)
=
z \left( \phi \right) - \ln \int_X \exp \left[ \phi - m_l \left( \phi \right) \right] \mathrm{d} \mu_k
\end{equation}
for all $\phi \in X^*$.
The significance of the above equation translates to the following fact about the function $V_{k,l} : \mathcal{K}(\mu_l) \to \mathbb{R}$ for $0 \le k \le l$ given by $T \mapsto \ln \int_X \exp[ T ] \mathrm{d} \mu_k$. 
\begin{lemma}
$V_{k,l}$ is finite everywhere, lower semicontinuous and norm-coercive, i.e.
\begin{equation}
\lim_{\Vert T \Vert \to \infty} V_{k,l}(T) = \infty \, .
\end{equation}
Moreover $\mathrm{int}\, \mathrm{dom}\, V_{k,l}^*$ is nonempty.
\end{lemma}
\begin{proof}
It is clear that $V_{k,l}$ is finite everywhere because $\exp[\phi] \in L^1(\mu)$ for all $\phi \in X^*$.
The lower semicontinuity follows immediately from the equivalence $\mu_k \sim \mu_l$ and Fatou's lemma.
Furthermore, since $\int_X T \mathrm{d} \mu_l = 0$ for all $T \in \mathcal{K}(\mu_l)$,
\begin{equation}
\begin{aligned}
\ln \int_X \exp[ T ] \mathrm{d} \mu_k
&\ge
\ln \frac{N_l}{N_k}
+
\ln \int_X \exp[ T ] \mathrm{d} \mu_l \\
&=
\ln \frac{N_l}{N_k}
+
\ln \int_X \left( \exp[ T ] - T \right) \mathrm{d} \mu_l
\ge
\ln \frac{N_l}{N_k}
+
\ln \int_X \max \left\{ 1, \left| T \right| \right\} \mathrm{d} \mu_l \, .
\end{aligned}
\end{equation}
It follows immediately from \cite[Exercise 2.41]{src:Zalinescu:ConvexAnalysisInGeneralVectorSpaces} that $0 \in \mathrm{int}\, \mathrm{dom}\, V_{k,l}^*$.
\end{proof}
Hence, instead of studying $V_k$ which may not be coercive at all, we can study $V_{k, l}$ which has the above coercivity property.
Also, since Wetterich's equation involves derivatives of the functions $V_k^*$, we should look for differentiability properties of $V_{k, l}^*$.
It is known that these are intimately connected to minimisation properties of the functions $V_{k,l}$.
However, the space $\mathcal{K}(\mu_l)$ is in general too small to exhibit minima such that we consider the function $\overline{W}_{k,l} : \mathcal{M}(\mu_l) \to \overline{\mathbb{R}}$ given as the largest lower semicontinuous function on $\mathcal{M}(\mu_l)$ that is smaller than $V_{k,l}$ on $\mathcal{K}(\mu_l)$.
Furthermore, we set $W_{k,l} = \overline{W}_{k,l} \upharpoonright \mathcal{L}(\mu_l)$.
\begin{lemma}
\label{lem:WProperties}
\begin{enumerate}
\item $V_{k,l}$, $W_{k,l}$ and $\overline{W}_{k,l}$ coincide on $\mathcal{K}(\mu_l)$,
\item $T \in \mathrm{dom}\, \overline{W}_{k,l} \implies \overline{W}_{k,l} (T) = \ln \int_X \exp[T] \mathrm{d} \mu_k$,
\item $T \in \mathrm{dom}\, \overline{W}_{k,l}  \implies T \in \mathcal{L}(\mu_l)$ and $\int_X T \mathrm{d} \mu_l \ge 0$,
\item $\mathrm{gr}\, V_{k,l}$ is dense in $\mathrm{gr}\, W_{k,l}$ with the subspace topology induced by $(\mathcal{L}(\mu_l), \overline{\tau_l}) \times \mathbb{R}$.
\end{enumerate}
\end{lemma}
\begin{proof}
$(i)$:
By definition, $\overline{W}_{k,l} \upharpoonright \mathcal{K}(\mu_l) \le V_{k,l}$ and since $\mathcal{M}(\mu)$ is first-countable,
\begin{equation}
\label{eq:WbarDefinition}
\begin{aligned}
\overline{W}_{k,l} \left( T \right)
&=
\inf \left\{ \liminf_{n \to \infty} V_{k,l} \left( T_n \right) : (T_n)_{n \in \mathbb{N}} \text{ in } \mathcal{K}(\mu_l) \text{ with } T_n \to T \text{ in $\mu_l$-measure} \right\} \\
&\ge
\ln \int_X \exp [ T ] \,\mathrm{d} \mu_k \, ,
\end{aligned}
\end{equation}
for all $T \in \mathcal{M}(\mu)$ by Fatou's lemma, since $\mu_k \sim \mu_l$.
Hence, $V_{k,l}$, $W_{k,l}$ and $\overline{W}_{k,l}$ coincide on $\mathcal{K}(\mu_l)$.

$(ii)$:
Clearly, $\mathrm{gr}\, V_{k,l}$ is dense in $\mathrm{gr}\, \overline{W}_{k,l}$.
Let $T \in \mathrm{dom}\, \overline{W}_{k,l}$ and $(T_n)_{n \in \mathbb{N}}$ be a sequence in $\mathrm{dom}\, V_{k,l}$ such that
\begin{equation}
T_n \to T \text{ in $\mu_l$-measure}
\qquad\text{and}\qquad
\lim_{n \to \infty} \ln \int_X \exp \left[ T_n \right] \mathrm{d} \mu_k = \overline{W}_{k,l} \left( T \right) \, .
\end{equation}
Then, $T_n$ also converges in $\mu_k$-measure such that for all $p > 1$, we have that $(\exp [T_n / p])_{n \in \mathbb{N}}$ is uniformly $\mu_k$-integrable and
\begin{equation}
\lim_{n \to \infty} \ln \int_X \exp \left[ \frac{T_n}{p} \right] \mathrm{d} \mu_k
=
\ln \int_X \exp \left[ \frac{T}{p} \right] \mathrm{d} \mu_k
<
\infty \, .
\end{equation}
By \cref{eq:WbarDefinition}, it follows that
\begin{equation}
\overline{W}_{k,l} \left( \frac{T}{p} \right) = \ln \int_X \exp \left[ \frac{T}{p} \right] \mathrm{d} \mu_k \, .
\end{equation}
Applying Hölders inequality and the lower semicontinuity of $W_{k,l}$,
\begin{equation}
\begin{aligned}
\overline{W}_{k,l} \left( T \right)
&\le
\liminf_{p \searrow 1} \overline{W}_{k,l} \left( \frac{T}{p} \right)
\le
\liminf_{p \searrow 1} \frac{1}{p} \ln \int_X \exp \left[ T \right] \mathrm{d} \mu_k \\
&=
\ln \int_X \exp \left[ T \right] \mathrm{d} \mu_k
\le
\overline{W}_{k,l} \left( T \right) \, .
\end{aligned}
\end{equation}

$(iii)$:
Let $T$ and $(T_n)_{n \in \mathbb{N}}$ be as above.
Then $( \max\{ T_n, 0 \} )_{n \in \mathbb{N}}$ is uniformly $\mu_k$-integrable and hence uniformly $\mu_l$-integrable such that
\begin{equation}
\lim_{n \to \infty} \int_X \max \left\{ T_n, 0 \right\} \mathrm{d} \mu_l
=
\int_X \max \left\{ T, 0 \right\} \mathrm{d} \mu_l \, .
\end{equation}
Moreover, $\int_X T_n \mathrm{d} \mu_l = 0$ such that by Fatou's lemma,
\begin{equation}
\int_X \max \left\{ T, 0 \right\} \mathrm{d} \mu_l
=
\lim_{n \to \infty} \int_X \max \left\{ - T_n, 0 \right\} \mathrm{d} \mu_l
\ge
\int_X \max \left\{ - T, 0 \right\} \mathrm{d} \mu_l \, .
\end{equation}
Hence, $T \in \mathcal{L}(\mu_l)$ and $\int_X T \mathrm{d} \mu_l \ge 0$.

$(iv)$:
Let $T$ and $(T_n)_{n \in \mathbb{N}}$ be as above.
Since $( \max\{ T_n, 0 \} )_{n \in \mathbb{N}}$ is uniformly $\mu_l$-integrable and $T \in L^1(\mu_l)$, $( \max\{ T_n - T, 0 \} )_{n \in \mathbb{N}}$ is uniformly 
$\mu_l$-integrable as well.
Hence,
\begin{equation}
\lim_{n \to \infty} \overline{p_l} \left( T_n - T \right)
=
\lim_{n \to \infty} \int_X \max \left\{ T_n - T, 0 \right\} \mathrm{d} \mu_l
=
0 \, .
\end{equation}
\end{proof}
As a consequence, $W_{k,l}$ is just as coercive as $V_{k,l}$ is.
\begin{corollary}
\label{cor:WklLowerBound}
$W_{k,l}$ is $L^1(\mu_l)$-norm coercive.
\end{corollary}
\begin{proof}
For all $T \in \mathrm{dom}\, W_{k,l}$, we have $\int_X T \mathrm{d} \mu_l \ge 0$ such that,
\begin{equation}
\begin{aligned}
W_{k,l} \left( T \right)
&=
\ln \int_X \exp \left[ T \right] \mathrm{d} \mu_k
\ge
\ln \frac{N_l}{N_k}
+
\ln \int_X \exp \left[ T \right] \mathrm{d} \mu_l \\
&\ge
\ln \frac{N_l}{N_k}
+
\ln \int_X \left( \exp \left[ T \right]- \int_X T \mathrm{d} \mu_l \right) \mathrm{d} \mu_l
\ge
\ln \frac{N_l}{N_k}
+
\ln \int_X \max\left\{ 1, \left| T \right| \right\} \mathrm{d} \mu_l \, .
\end{aligned}
\end{equation}
\end{proof}
As we shall see, $\mathcal{L}(\mu_l)$ does indeed exhibit the desired minima and we just need to translate from the domain of $V_k^*$, i.e. $\mathcal{A}(\mu)$, to the domain of $W_{k,l}^*$, i.e. $(\mathcal{L}(\mu_l), \overline{\tau_l})^*$.
\begin{lemma}
For every $y \in \mathcal{A}(\mu)$ and $k \ge 0$, the linear map $\mathcal{K}(\mu_k) \to \mathbb{R}$ given by
\begin{equation}
\left[ \phi - m_k \left( \phi \right) \right]_{\mu_k}
\mapsto
\left( y - m_k \right) \left( \phi  \right)
\end{equation}
for all $\phi \in X^*$ is well-defined and extends continuously to a unique element $M_k(y) \in \mathcal{M}(\mu_k)^*$.
\end{lemma}
\begin{proof}
From \cref{eq:AffineKernelInMeanShift}, it is clear that $y - m_k \in \mathcal{H}(\nu_k)$.
Moreover, for any $\phi, \psi \in X^*$ with $[ \phi - m_k ( \phi ) ]_{\mu_k} = [ \psi - m_k ( \psi ) ]_{\mu_k}$, we have $\phi = \psi$ $\nu_k$-almost everywhere and thus $(y - m_k)( \phi ) = (y - m_k)( \psi )$.
Now, let $(T_n)_{n \in \mathbb{N}}$ be a sequence in $\mathcal{K}(\mu_k)$ that is Cauchy in $\mathcal{M}(\mu_k)$.
Then, every corresponding sequence $(\phi_n)_{n \in \mathbb{N}}$ in $X^*$ with $T_n = [\phi_n - m_k(\phi_n)]_{\mu_k}$ is clearly Cauchy with respect to the topology of convergence in $\nu_k$-measure.
Consequently, $(y - m_k)( \phi_n )$ is Cauchy and the claim follows.
\end{proof}
Since $\mathcal{M}(\mu_k)^* \subseteq (\mathcal{L}(\mu_l), \overline{\tau_l})^*$, it is easy to see that $W_{k,l}$ encodes precisely the right information.
\begin{corollary}
\label{cor:VkAstEqualsVklAstEqualsWklAst}
For all $y \in \mathcal{A}(\mu)$ and all $0 \le k \le l$, $V_k^*(y) = V_{k,l}^*(M_l(y)) = W_{k,l}^*(M_l(y))$.
\end{corollary}
\begin{proof}
Pick $z_l \in \mathcal{H}(\nu_l)$ such that $y = m_l + z_l$.
Then
\begin{equation}
\begin{aligned}
V_k^* \left( y \right)
&=
\sup_{\phi \in X^*} \left[ y \left( \phi \right) - \ln \int_X \exp \left[ \phi \right] \mathrm{d} \mu_k \right] \\
&=
\sup_{\phi \in X^*} \left[ M_l \left( y \right) \left( \left[ \phi - m_l \left( \phi \right) \right]_{\mu_l} \right) - \ln \int_X \exp \left[ \phi - m_l \left( \phi \right) \right] \mathrm{d} \mu_k \right] \\
&=
\sup_{T \in \mathcal{K}(\mu_l)} \left[ M_l \left( y \right) \left( T \right) - V_{k,l}( T ) \right]
=
V_{k,l}^* \left( M_l \left( y \right) \right) \, .
\end{aligned}
\end{equation}
That $V_{k,l}^*(M_l(y)) = W_{k,l}^*(M_l(y))$, follows since $\mathrm{gr}\, [ M_l(y) - V_{k,l} ]$ is dense in $\mathrm{gr}\, [ M_l(y) - W_{k,l} ]$ in the subspace topology of $\mathcal{M}(\mu_l) \times \mathbb{R}$.
\end{proof}
The underlying reason to use the $\tau_l$ topology and not just the $p_l$-topology boils down to the following \namecref{thm:L1CoercivityOnInterior} for which it is imperative, that one can, for certain $\beta \in (\mathcal{L}(\mu_l), \overline{\tau_l})^*$ and for every $T \in \mathcal{L}(\mu_l)$, approximate $\beta(T)$, by a sequence $\beta(S_n)$ with $S_n \in \mathcal{K}(\mu_l)$.
\begin{theorem}
\label{thm:L1CoercivityOnInterior}
Let $0 \le k \le l$ and $\beta \in \mathcal{M}(\mu_l)^*$.
Then the following are equivalent,
\begin{enumerate}
\item $\beta \in \mathrm{int}\, \mathrm{dom}\, V_{k,l}^*$
\item There is some $\epsilon > 0$ such that
\begin{equation}
\label{eq:pBarDualInteriorEstimate}
\sup \left\{ W_{k,l}^* \left( \beta + \alpha \right) \middle| \alpha \in \left( \mathcal{L}(\mu_l), \overline{p_l} \right)^* : \bar{p}_l^* \left( \alpha \right) < \epsilon \right\} < \infty \, .
\end{equation}
\item $W_{k,l} - \beta$ is $L^1(\mu_l)$-norm coercive.
\end{enumerate}
\end{theorem}
\begin{proof}
$(i) \implies (ii)$:
Let $\Vert \cdot \Vert_l^*$ denote the dual norm on $\mathcal{K}(\mu_l)^*$ and define $\overline{B_\epsilon^*( 0 )}$ to be the closed $\Vert \cdot \Vert_l^*$-ball of radius $\epsilon > 0$ around zero.
Then, by \cite[Exercise 2.45(a)]{src:Zalinescu:ConvexAnalysisInGeneralVectorSpaces}, there is some $\epsilon > 0$ such that
\begin{equation}
\sup_{\alpha \in \overline{B_\epsilon^*( 0 )}} V_{k,l}^* \left( \beta + \alpha \right) < \infty \, .
\end{equation}
Now, let $\alpha \in ( \mathcal{L}(\mu_l), \overline{p_l} )^*$ with $\bar{p}_l^* ( \alpha ) < \epsilon$ and let $(T_n)_{n \in \mathbb{N}}$ be any sequence in $\mathcal{L}(\mu_l)$ that maximises $\beta + \alpha - W_{k,l}$.
Clearly,
\begin{equation}
\beta \left( T_n \right) + \alpha \left( T_n \right) - W_l \left( T_n \right)
\le
\beta \left( T_n \right) + \overline{p_l}^* \left( \alpha \right) \left\Vert T_n \right\Vert_{L^1(\mu_l)} - W_l \left( T_n \right) \, .
\end{equation}
For each $n \in \mathbb{N}$, let $(S^n_m)_{m \in \mathbb{N}}$ be a sequence in $\mathcal{K}(\mu_l)$ such that
\begin{equation}
S^n_m \text{ converges to } T_n \text{ in $\mu_l$-measure and } \lim_{m \to \infty} V_{k,l} \left( S^n_m \right) = W_{k,l} \left( T_n \right) \, .
\end{equation}
Then, by Fatou's lemma,
\begin{equation}
\liminf_{m \to \infty} \left\Vert S^n_m \right\Vert_{L^1(\mu_l)}
\ge
\left\Vert T_n \right\Vert_{L^1(\mu_l)} \, .
\end{equation}
Consequently, we find some diagonal sequence $(S_n)_{n \in \mathbb{N}}$ in $\mathcal{K}(\mu_l)$ such that for each $n \in \mathbb{N}$,
\begin{equation}
\beta \left( T_n \right) \le \beta \left( S_n \right) + \frac{1}{n} \, , \quad
W_{k,l} \left( T_n \right) \ge W_{k,l} \left( S_n \right) - \frac{1}{n} \, , \quad
\left\Vert T_n \right\Vert_{L^1(\mu_l)} \le \left\Vert S_n \right\Vert_{L^1(\mu_l)} + \frac{1}{n} \, .
\end{equation}
Hence,
\begin{equation}
\begin{aligned}
\beta \left( T_n \right) + \alpha \left( T_n \right) - W_l \left( T_n \right)
&\le
\beta \left( S_n \right) + \overline{p_l}^* \left( \alpha \right) \left\Vert S_n \right\Vert_{L^1(\mu_l)} - V_{k,l} \left( S_n \right) + \frac{3}{n} \\
&\le
\sup_{S \in \mathcal{K}(\mu_l)} \left[ \beta  \left( S \right) + \epsilon \left\Vert S \right\Vert_{L^1(\mu_l)} - V_{k,l} \left( S \right) \right] + \frac{3}{n} \\
&=
\sup_{S \in \mathcal{K}(\mu_l)} \sup_{\gamma \in \overline{B_\epsilon^*( 0 )}} \left[ \beta \left( S \right) + \gamma \left( S \right) - V_{k,l} \left( S \right) \right] + \frac{3}{n} \\
&=
\sup_{\gamma \in \overline{B_\epsilon^*( 0 )}} V_{k,l}^* \left( \beta + \gamma \right) + \frac{3}{n} \, .
\end{aligned}
\end{equation}
Since this bound is independent of $\alpha$, \cref{eq:pBarDualInteriorEstimate} holds.

$(ii) \implies (iii)$:
It is well-known that $\sup_{\overline{p_l}^*(\alpha) \le 1} \alpha(T) = \overline{p_l}(T)$ for all $T \in \mathcal{L}(\mu_l)$ (see e.g. \cite[Corollary 2.2.4]{src:Cobzaş:FunctionalAnalysisinAsymmetricNormedSpaces} and note that the case $\overline{p_l}(T) = 0$ is trivial).
Furthermore, recalling that $T \in \mathrm{dom}\, W_{k,l}$ implies $\int_X T \mathrm{d} \mu_l \ge 0$, we also have $\overline{p_l}(T) \ge \Vert T \Vert_{L^1(\mu_l)} / 2$.
Hence,
\begin{equation}
\begin{aligned}
\sup &\left\{ W_{k,l}^* \left( \beta + \alpha \right) \middle| \alpha \in \left( \mathcal{L}(\mu_l), \overline{p_l} \right)^* : \bar{p}_l^* \left( \alpha \right) < \epsilon \right\}
=
\sup_{T \in \mathcal{L}(\mu_l)} \left[ \beta \left( T \right) + \epsilon \overline{p_l} \left( T \right) - W_{k,l} \left( T \right) \right] \\
&\ge
\sup_{T \in \mathcal{L}(\mu_l)} \left[ \beta \left( T \right) + \frac{\epsilon}{2} \left\Vert T \right\Vert_{L^1(\mu_l)} - W_{k,l} \left( T \right) \right] \, .
\end{aligned}
\end{equation}
Since the above is finite, we have that $W_{k,l} - \beta$ is $L^1(\mu_l)$-norm coercive.

$(iii) \implies (i)$:
Clearly, $V_{k,l} - M_l(y)$ is also $L^1(\mu_l)$-norm coercive.
Hence, the claim follows from \cite[Exercise 2.41]{src:Zalinescu:ConvexAnalysisInGeneralVectorSpaces}.
\end{proof}
Finally, the minimisation properties follow, inspired by \cite[Theorem 5.3]{src:FischerZiebell:Tychonov}.
\begin{lemma}
\label{lem:WklMinimisation}
Let $0 \le k \le l$ and $\beta \in \mathcal{M}(\mu_l)^*$ such that $\beta \in \mathrm{int}\, \mathrm{dom}\, V_{k,l}^*$.
Then the set of minimisers $M$ of $W_{k,l} - \beta$ is nonempty and has the form $M = \{ S + c : c \in I \}$ for some $S \in \mathcal{L}(\mu_l)$ and a compact interval $I \subset \mathbb{R}$.
Moreover, for every minimising sequence $(T_n)_{n \in \mathbb{N}}$, there is some $S \in M$ and a subsequence $(R_n)_{n \in \mathbb{N}}$ such that for every further subsequence $(R'_n)_{n \in \mathbb{N}}$,
\begin{equation}
A_n = \frac{1}{n} \sum_{m = 1}^n R'_m
\end{equation}
$\overline{\tau_l}$-converges to $S$.
\end{lemma}
\begin{proof}
Letting $(T_n)_{n \in \mathbb{N}}$ be a minimising sequence, \cref{thm:L1CoercivityOnInterior} shows that $T_n$ is $L^1(\mu_l)$-bounded.
Consequently, by Komlòs' theorem, there is some $S \in L^1(\mu_l)$ and a subsequence $(R_n)_{n \in \mathbb{N}}$ such that for every further subsequence $(R'_n)_{n \in \mathbb{N}}$,
\begin{equation}
A_n = \frac{1}{n} \sum_{m = 1}^n R'_m
\end{equation}
converges in $\mu_l$-measure to $S$.
By construction, it is clear that $S \in \mathcal{L}(\mu_l)$ and by convexity, $(A_n)_{n \in \mathbb{N}}$ is a minimising sequence of $W_{k,l} - \beta$.
Without loss of generality, $A_n \in \mathrm{dom}\, W_{k,l}$ such that $\max\{ A_n, 0 \}$ is uniformly $\mu_l$-integrable and consequently, $\max\{ A_n - S, 0 \}$ is uniformly $\mu_l$-integrable as well.
Hence, by Vitali's convergence theorem,
\begin{equation}
\lim_{n \to \infty} \overline{p_l} \left( A_n - S \right)
=
\lim_{n \to \infty} \int_X \max\left\{ A_n - S, 0 \right\} \mathrm{d} \mu_l
=
0 \, ,
\end{equation}
i.e. $A_n$ $\overline{\tau_l}$-converges to $S$.
Because $\lim_{n \to \infty} \beta( A_n ) = \beta(S)$ it follows from the lower semicontinuity of $W_{k,l}$ with respect to convergence in measure, that $S$ is a minimiser of $W_{k,l} - \beta$, i.e. $S \in M$.
Now, let $S' \in M$ and note that the function
\begin{equation}
f : [0,1] \to \mathbb{R}, t \mapsto W_{k,l} \left( \left[ 1 - t \right] S + t S' \right)
=
\ln \int_X \exp \left[ \left( 1 - t \right) S + t S' \right] \mathrm{d} \mu_k \, .
\end{equation}
is twice differentiable on the open interval $(0,1)$ such that for some $t \in (0,1)$,
\begin{equation}
W_{k,l} \left( S' \right) - \beta \left( S' \right)
=
W_{k,l} \left( S \right) - \beta \left( S \right)
+
\frac{1}{2} f'' \left( t \right) \, ,
\end{equation}
i.e. $f'' ( t ) = 0$.
Writing $U = (1 - t) S + t S'$ for brevity, we obtain
\begin{equation}
f'' \left( t \right)
=
\frac{\int_X \left( S' - S \right)^2 \exp \left[ U \right] \mathrm{d} \mu_k}{\int_X \exp \left[ U \right] \mathrm{d} \mu_k}
- \left( \frac{\int_X \left( S' - S \right) \exp \left[ U \right] \mathrm{d} \mu_k}{\int_X \exp \left[ U \right] \mathrm{d} \mu_k} \right)^2
=
0 \, .
\end{equation}
Since this amounts to Hölder's inequality with the function $S' - S$ and the constant function $1$ becoming an equality, we conclude that that $S'$ and $S$ differ by a constant $\mu_k$-almost everywhere and thus also $\mu_l$-almost everywhere.
Hence, by the convexity of $W_{k,l} - \beta$, there exists an interval $I \subseteq \mathbb{R}$ such that
\begin{equation}
M = \left\{ S + c : c \in I \right\}
\end{equation}
and since $W_{k,l} - \beta$ is $L^1(\mu)$-coercive, $I$ is bounded.
Moreover, by the lower semicontinuity of $W_{k,l} - \beta$ with respect to convergence in measure, it is immediate that $\inf I \in I$ and $\sup I \in I$.
\end{proof}
\begin{corollary}
\label{cor:WklMinimisation}
Let $0 \le k \le l$ and $\beta \in \mathcal{M}(\mu_l)^*$ such that such that $\beta \in \mathrm{int}\, \mathrm{dom}\, V_{k,l}^*$.
Then, there is a $T \in \mathcal{L}(\mu_l)$ minimising $W_{k,l} - \beta$ for which every minimising sequence of $W_{k,l} - \beta$ converges $\bar{p}_l$-weakly to $T$.
\end{corollary}
\begin{proof}
Let $(T_n)_{n \in \mathbb{N}}$ be a minimising sequence.
Using \cref{lem:WklMinimisation}, there is a unique greatest element $T$ in the set of minimisers $M$ of $W_{k,l} - \beta$.
Furthermore, for every subsequence $(T_n)_{n \in \mathbb{N}}$ there is a $S \in M$ and a subsequence $(R_n)_{n \in \mathbb{N}}$ of $(T_n)_{n \in \mathbb{N}}$ such that for every further subsequence $(R'_n)_{n \in \mathbb{N}}$,
\begin{equation}
A_n = \frac{1}{n} \sum_{m = 1}^n R'_m
\end{equation}
$\overline{\tau_l}$-converges to $S$.
Since $T \ge S$, we clearly have that $A_n$ $\overline{p_l}$-converges to $T$ as well.
Consequently, for every $\alpha \in (\mathcal{L}(\mu_l), \overline{p_l})^*$, there is some $C \ge 0$ such that
\begin{equation}
\limsup_{n \to \infty} \alpha \left( A_n - T \right)
\le
C \limsup_{n \to \infty} \overline{p_l} \left( A_n - T \right)
=
0 \, ,
\end{equation}
i.e. $\limsup_{n \to \infty} \alpha ( A_n ) \le \alpha( T )$.
Hence, \cref{lem:WeakConvergence} applies and $R_n$ converges $\overline{p_l}$-weakly to $T$.
Consequently, every subsequence of $(T_n)_{n \in \mathbb{N}}$ has a subsequence converging $\overline{p_l}$-weakly to $T$ such that, indeed, $T_n$ converges $\overline{p_l}$-weakly to $T$.
\end{proof}
Applying \cref{thm:DualDifferentiability}, we also obtain the following.
\begin{corollary}
\label{cor:ConjugateDiffability}
Let $0 \le k \le l$ and $\beta \in \mathcal{M}(\mu_l)^*$ be such that $\beta \in \mathrm{int}\, \mathrm{dom}\, V_{k,l}^*$ and $T_k \in \mathcal{L}(\mu_l)$ satisfies the conclusion of \cref{cor:WklMinimisation}.
Then, for every $\alpha \in (\mathcal{L}(\mu_l), \overline{p_l})^*$,
\begin{equation}
\lim_{ t  \searrow 0 } \frac{1}{t} \left[ W_{k,l}^* \left( \beta + t \alpha \right) - W_{k,l}^* \left( \beta \right) \right]
=
\alpha \left( T_k \right) \, .
\end{equation}
\end{corollary}
Hence, $W_{k,l}^*$ is differentiable at such points $\beta$ in the given sense.
It is however, useful to show that is, in fact, also subdifferentiable in a very natural sense.
To that end we shall consider the function $W_{k,l}'$ on the Banach space $(\mathcal{L}(\mu_l), L^1(\mu_l))$ and note that $(W_{k,l}')^*$ coincides with $W_{k,l}^*$ on the intersection of their domains.
The following \namecref{lem:WPrimeSubdifferentialRepresentation} shows that subdifferentials of $W_{k,l}'$ are comparatively easy to understand.
\begin{lemma}
\label{lem:WPrimeSubdifferentialRepresentation}
Let $T \in \mathrm{dom}\, \partial W'_{k,l}$.
For every $\alpha \in \partial W'_{k,l}(T)$ and all $S \in \mathcal{K}(\mu_l)$,
\begin{equation}
\alpha \left( S \right)
=
\frac{\int_X S \exp \left[ T \right] \mathrm{d} \mu_k}{\int_X \exp \left[ T \right] \mathrm{d} \mu_k}
\qquad\text{and}\qquad
\alpha \left( T \right)
\ge
\frac{\int_X T \exp \left[ T \right] \mathrm{d} \mu_k}{\int_X \exp \left[ T \right] \mathrm{d} \mu_k} \, .
\end{equation}
\end{lemma}
\begin{proof}
Let $\alpha \in \partial W'_{k,l}(T)$ and note that since $0 \in \mathrm{dom}\, W'_{k,l}$, for all $t \in (0,1)$,
\begin{equation}
\begin{aligned}
\alpha \left( T \right)
&\ge
\frac{1}{t} \left[ \ln \int_X \exp \left[ T \right] \mathrm{d} \mu_k - \ln \int_X \exp \left[ T - t T \right] \mathrm{d} \mu_k \right] \\
&=
\frac{1}{t} \ln \frac{\int_X \exp \left[ t T \right] \exp \left[ \left( 1 - t \right) T \right] \mathrm{d} \mu_k}{\int_X \exp \left[ \left( 1 - t \right) T \right] \mathrm{d} \mu_k}
\ge
\frac{\int_X T \exp \left[ \left( 1 - t \right) T \right] \mathrm{d} \mu_k}{\int_X \exp \left[ \left( 1 - t \right) T \right] \mathrm{d} \mu_k} \, .
\end{aligned}
\end{equation}
The last estimate follows from Jensen's inequality.
Using the monotone convergence theorem over the set $\{ T \ge 0 \}$ and the dominated convergence theorem over $\{ T < 0 \}$, it is then evident that
\begin{equation}
\alpha \left( T \right)
\ge
\lim_{t \searrow 0} \frac{\int_X T \exp \left[ \left( 1 - t \right) T \right] \mathrm{d} \mu_k}{\int_X \exp \left[ \left( 1 - t \right) T \right] \mathrm{d} \mu_k}
=
\frac{\int_X T \exp \left[ T \right] \mathrm{d} \mu_k}{\int_X \exp \left[ T \right] \mathrm{d} \mu_k} \, .
\end{equation}
In particular, the right-hand side is finite.
Now we take a short detour and define the two Young functions \cite[p. 15]{src:RaoRen:OrliczSpaces}
\begin{equation}
\Phi \left( x \right) = e^{\left| x \right|} - \left| x \right| - 1
\qquad\text{and}\qquad
\Psi \left( y \right) = \left( 1 + \left| y \right| \right) \ln \left( 1 + \left| y \right| \right) - \left| y \right|
\end{equation}
for all $x, y \in \mathbb{R}$.
Furthermore, consider the associated Orlicz norms for real-valued $\mu_k$-measurable functions $f$ and $g$ given by
\begin{equation}
N^k_\Phi \left( f \right)
=
\inf \left\{ r > 0 : \int_X \Phi \left( \frac{f}{r} \right) \mathrm{d} \mu_k \le 1 \right\}
,\,
N^k_\Psi \left( g \right)
=
\inf \left\{ r > 0 : \int_X \Psi \left( \frac{g}{r} \right) \mathrm{d} \mu_k \le 1 \right\} \, .
\end{equation}
Since $\Phi$ and $\Psi$ are convex, crude estimates for these are simply given by
\begin{equation}
N^k_\Phi \left( f \right)
\le
\max\left\{ 1, \int_X \Phi \left( f \right) \mathrm{d} \mu_k \right\}
\qquad\text{and}\qquad
N^k_\Psi \left( g \right)
\le
\max\left\{ 1, \int_X \Psi \left( g \right) \mathrm{d} \mu_k \right\} \, .
\end{equation}
Now, letting $S \in \mathcal{K}(\mu_l)$, we use the Orlicz-Hölder inequality (see e.g. \cite[p.58, Proposition 1, Remark]{src:RaoRen:OrliczSpaces}) to deduce that
\begin{equation}
\int_X \left| S \right| \exp \left[ T \right] \mathrm{d} \mu_k
\le
2 N^k_\Phi \left( S \right) N^k_\Psi \left( \exp \left[ T \right] \right) \, .
\end{equation}
Noting that $\exp[S] + \exp[-S]$ is $\mu_k$-integrable, we see that $\int_X \exp[ |S| ] \mathrm{d} \mu_k$ is finite and hence $N^k_\Phi(S) < \infty$.
Likewise, note that $\Psi(y) \le 1 + y \exp[y]$ for all $y \in \mathbb{R}$ such that,
\begin{equation}
\int_X \Psi \left( \exp \left[ T \right] \right) \mathrm{d} \mu_k
\le
1 + \int_X T \exp \left[ T \right] \mathrm{d} \mu_k
<
\infty \, ,
\end{equation}
such that $N^k_\Psi ( \exp [ T ] ) < \infty$.
Consequently, we may define the linear functional $\beta : \mathrm{span}\, ( \{ T \} \cup \mathcal{K}(\mu_l) ) \to \mathbb{R}$ with
\begin{equation}
R \mapsto \frac{\int_X R \exp \left[ T \right] \mathrm{d} \mu_k}{\int_X \exp \left[ T \right] \mathrm{d} \mu_k} \, .
\end{equation}
Now, consider the family $R_t = (1-t) T + t S$ and note that since $\alpha$ is a subdifferential,
\begin{equation}
\begin{aligned}
\alpha \left( S - T \right)
&=
\frac{1}{t} \alpha \left( R_t - T \right)
\le
\frac{1}{t} \left[ W_{k,l} \left( R_t \right) - W_{k,l} \left( T \right) \right] \\
&=
- \frac{1}{t} \ln \frac{ \int_X \exp \left[ t \left( T - S \right) \right] \exp \left[ \left( 1 - t \right) T + t S \right] \mathrm{d} \mu_k }{\int_X \exp \left[ \left( 1 - t \right) T + t S \right] \mathrm{d} \mu_k} \\
&\le
\frac{ \int_X \left( S - T \right) \exp \left[ \left( 1 - t \right) T + t S \right] \mathrm{d} \mu_k }{\int_X \exp \left[ \left( 1 - t \right) T + t S \right] \mathrm{d} \mu_k} \, ,
\end{aligned}
\end{equation}
where the last estimate follows from Jensen's inequality.
We want to show that the right-hand side tends to $\beta(S - T)$ as $t \searrow 0$.
The denominator is continuous for $t \in [0,1]$ and poses no problem.
By convexity, $| S | \exp[ (1-t)T + tS ] \le |S| \exp[T] + |S| \exp[S]$, such that
\begin{equation}
\lim_{ t  \searrow 0 } \int_X S \exp \left[ \left( 1 - t \right) T + t S \right] \mathrm{d} \mu_k
=
\int_X S \exp \left[ T \right] \mathrm{d} \mu_k \, ,
\end{equation}
by the dominated convergence theorem.
For the remaining part, note that, by Hölder's inequality,
\begin{equation}
\begin{aligned}
\int_X \max \left\{ T, 0 \right\} \exp \left[ S \right] \mathrm{d} \mu_k
&\le
\int_X \exp \left[ S + \frac{1}{2} \max \left\{ T, 0 \right\} \right] \mathrm{d} \mu_k \\
&\le
\sqrt{ \int_X \exp \left[ 2 S \right] \mathrm{d} \mu_k \int_X \left( 1 + \exp \left[ T \right] \right) \mathrm{d} \mu_k }
<
\infty \, .
\end{aligned}
\end{equation}
Hence, using $\exp[ (1-t)T + tS ] \le \exp[T] + \exp[S]$ and the dominated convergence theorem again, we find that
\begin{equation}
\lim_{ t  \searrow 0 } \int_X \max \left\{ T, 0 \right\} \exp \left[ \left( 1 - t \right) T + t S \right] \mathrm{d} \mu_k
=
\int_X \max \left\{ T, 0 \right\} \exp \left[ T \right] \mathrm{d} \mu_k \, .
\end{equation}
Finally, for $t < 1/2$, we have $\max \{ -T, 0 \} \exp [ ( 1 - t ) T + t S ] \le \exp [ \max\{ S, 0 \} ]$ such that
\begin{equation}
\lim_{ t  \searrow 0 } \int_X \max \left\{ - T, 0 \right\} \exp \left[ \left( 1 - t \right) T + t S \right] \mathrm{d} \mu_k
=
\int_X \max \left\{ - T, 0 \right\} \exp \left[ T \right] \mathrm{d} \mu_k \, .
\end{equation}
Consequently, for all $S \in \mathcal{K}(\mu_l)$,
\begin{equation}
\begin{aligned}
\alpha \left( S - T \right)
\le
\frac{ \int_X \left( S - T \right) \exp \left[ T \right] \mathrm{d} \mu_k }{\int_X \exp \left[ T \right] \mathrm{d} \mu_k}
=
\beta \left( S - T \right) \, .
\end{aligned}
\end{equation}
It follows that $(\alpha - \beta)(S) \le (\alpha - \beta)(T)$ i.e. the linear functional $(\alpha - \beta) \upharpoonright \mathcal{K}(\mu_l)$ is majorised by a constant.
Evidently, that can only happen if $\alpha$ and $\beta$ coincide on $\mathcal{K}(\mu_l)$.
\end{proof}
The striking consequence is that suitable $\beta \in \mathrm{int}\, \mathrm{dom}\, V_{k,l}^*$ have simple integral expressions.
\begin{theorem}
\label{thm:MlyRepresentation}
Let $0 \le k \le l$ and $\beta \in \mathcal{M}(\mu_l)^*$ be such that $\beta \in \mathrm{int}\, \mathrm{dom}\, V_{k,l}^*$ and $T_k \in \mathcal{L}(\mu_l)$ satisfies the corresponding conclusion of \cref{cor:WklMinimisation}.
Then, $T_k \in \partial (W'_{k,l})^*( \beta )$ and for all $S \in \mathcal{K}(\mu_l)$,
\begin{equation}
\label{eq:MlyRepresentation}
\beta \left( S \right)
=
\frac{\int_X S \exp \left[ T_k \right] \mathrm{d} \mu_k}{\int_X \exp \left[ T_k \right] \mathrm{d} \mu_k}
\qquad\text{and}\qquad
\beta \left( T_k \right)
\ge
\frac{\int_X T_k \exp \left[ T_k \right] \mathrm{d} \mu_k}{\int_X \exp \left[ T_k \right] \mathrm{d} \mu_k}
\, .
\end{equation}
\end{theorem}
\begin{proof}
Let $\gamma \in ( \mathcal{L}(\mu_l), L^1(\mu_l) )^*$ and note that by definition,
\begin{equation}
(W'_{k,l})^* \left( \beta \right)
+
\gamma \left( T_k \right)
=
\left( \beta + \gamma \right) \left( T_k \right)
-
W'_{k,l} \left( T_k \right)
\le
(W'_{k,l})^* \left( \beta + \gamma \right) \, .
\end{equation}
Hence, $T_k \in \partial (W'_{k,l})^*( \beta )$ and since $W_{k,l}'$ is lower semicontinuous it follows that $\beta \in \partial W'_{k,l}(T_k)$ (see e.g. \cite[Theorem 2.4.4(iv)]{src:Zalinescu:ConvexAnalysisInGeneralVectorSpaces}).
Now, \cref{eq:MlyRepresentation} follows from \cref{lem:WPrimeSubdifferentialRepresentation}.
\end{proof}
Having understood these minimisation properties of $W_{k,l} - \beta$, we just need the following equi-coercivity property to derive a differential equation for $V_k^*(y)$.
\begin{lemma}
\label{lem:EquiInterior}
Let $0 \le k \le l$ and $y \in \mathcal{A}(\mu)$ with $M_l(y) \in \mathrm{int}\, \mathrm{dom}\, V_{k,l}^*$.
Then, for each $a \in [0, k]$, $M_l(y) \in \mathrm{int}\, \mathrm{dom}\, V_{a,l}^*$.
Furthermore, there exist $A > 0$ and $B \in \mathbb{R}$ such that for all $a \in [0, k]$ and all $T \in \mathcal{L}(\mu_l)$,
\begin{equation}
\label{eq:EquiNormCoercivity}
W_{a,l} \left( T \right) - M_l(y) \left( T \right)
\ge
A \left \Vert T \right\Vert_{L^1(\mu_l)} + B \, .
\end{equation}
Moreover,
\begin{equation}
\label{eq:EquiEstimate}
\sup_{a \in [0, k]} \frac{N^a_\Psi \left( \exp \left[ T_a \right] \right)}{\int_X \exp \left[ T_a \right] \mathrm{d} \mu_a}
<
\infty \, .
\end{equation}
\end{lemma}
\begin{proof}
Applying \cref{thm:L1CoercivityOnInterior}, we find that $W_{k,l} - M_l(y)$ is $L^1(\mu_l)$-norm coercive.
Hence, by \cite[Exercise 2.41]{src:Zalinescu:ConvexAnalysisInGeneralVectorSpaces}, there exist $A > 0$ and $C \in \mathbb{R}$ such that
\begin{equation}
W_{k,l} \left( T \right) - M_l(y) \left( T \right)
\ge
A \left \Vert T \right\Vert_{L^1(\mu_l)} + C \, .
\end{equation}
Now, for all $a \in [0, k]$ and $T \in \mathcal{K}(\mu_l)$,
\begin{equation}
V_{a,l} \left( T \right)
=
\ln \int_X \exp \left[ T \right] \mathrm{d} \mu_a
\ge
\ln \frac{N_k}{N_a} + V_{k,l} \left( T \right) \, .
\end{equation}
By construction, it follows that $W_{a,l} \ge W_{k,l} + \ln (N_k/N_a)$ and thus
\begin{equation}
W_{a,l} \left( T \right) - M_l(y) \left( T \right)
\ge
W_{k,l} \left( T \right) - M_l(y) \left( T \right) + \ln \frac{N_k}{N_a}
\ge
A \left \Vert T \right\Vert_{L^1(\mu_l)} + C + \ln N_k \, ,
\end{equation}
for all $T \in \mathcal{L}(\mu_l)$.
That $M_l(y) \in \mathrm{int}\, \mathrm{dom}\, V_{a,l}^*$ follows from \cref{thm:L1CoercivityOnInterior}.

To verify \cref{eq:EquiEstimate}, note that the denominator $\int_X \exp [ T_a ] \mathrm{d} \mu_a$ is bounded away from zero, since
\begin{equation}
\int_X \exp \left[ T_a \right] \mathrm{d} \mu_a
\ge
\int_X \exp \left[ T_a \right] \mathrm{d} \mu_l \frac{N_l}{N_a}
\ge
\int_X \exp \left[ T_a - \int_X T_a \mathrm{d} \mu_l \right] \mathrm{d} \mu_l \frac{N_l}{N_a}
\ge
N_l \, .
\end{equation}
The second inequality follows because $T_a \in \mathrm{dom}\, W_{a,l}$ implies $\int_X T_a \mathrm{d} \mu_l \ge 0$ and the third inequality by the same reasoning as in \cref{cor:WklLowerBound}.
Furthermore,
\begin{equation}
N^a_\Psi \left( \exp \left[ T_a \right] \right)
\le
1
+
\int_X \Psi \left( \exp \left[ T_a \right] \right) \mathrm{d} \mu_a
\le
2
+
\int_X T_a \exp \left[ T_a \right] \mathrm{d} \mu_a \, .
\end{equation}
Combining the above and \cref{thm:MlyRepresentation}, we arrive at
\begin{equation}
\frac{N^a_\Psi \left( \exp \left[ T_a \right] \right)}{\int_X \exp \left[ T_a \right] \mathrm{d} \mu_a}
\le
M_l(y) \left( T_a \right) + \frac{2}{N_l} \, ,
\end{equation}
Suppose that $(T_a)_{a \in [0, k]}$ is $L^1(\mu_l)$-unbounded.
Then, \cref{eq:EquiNormCoercivity} shows that the function $a \mapsto W_{a,l}^*( M_l(y) )$ is unbounded on $[0,k]$.
However, by \cref{cor:VkAstEqualsVklAstEqualsWklAst,thm:VAstContinuity}, this is a contradiction.
Hence, $(T_a)_{a \in [0, k]}$ is $L^1(\mu_l)$-bounded such that by continuity, $M_l ( y ) ( T_a )$ is bounded as well.
\end{proof}
Now, we can finally prove that $k \mapsto V_k^*(y)$ satisfies a differential equation.
\begin{theorem}
Let $0 \le k \le l$ and $y \in \mathcal{A}(\mu)$ be such that $M_l(y) \in \mathrm{int}\, \mathrm{dom}\, V_{k,l}^*$ and for all $a \in [0, k]$, let $T_a \in \mathcal{L}(\mu_l)$ be as in \cref{cor:WklMinimisation}.
Then, the function $V^*(y) : [0, k] \to \mathbb{R}, a \mapsto V_a^*(y)$ is absolutely continuous and for almost every $a \in [0, k]$,
\begin{equation}
\label{eq:VAstDerivative}
\frac{\mathrm{d}}{\mathrm{d} a} V_a^* \left( y \right)
=
\frac{\int_X Q'_a \exp \left[ T_a \right] \mathrm{d} \mu_a}{\int_X \exp \left[ T_a \right] \mathrm{d} \mu_a}
-
\frac{1}{2} \int_X Q_a' \mathrm{d} \mu_a \, .
\end{equation}
\end{theorem}
\begin{proof}
By \cref{lem:EquiInterior,thm:MlyRepresentation}, for each $a \in [0,k]$, there exists $T_a \in \mathcal{L}(\mu_l)$ such that \cref{eq:MlyRepresentation} holds and these $T_a$ minimise $W_{a,l} - M_l(y)$.
Moreover, by \cref{cor:VkAstEqualsVklAstEqualsWklAst}, we have
\begin{equation}
V_a^* \left( y \right) = M_l \left( y \right) \left( T_a \right) - W_{a,l} \left( T_a \right) \, .
\end{equation}
Pick a sequence $(S^a_n)_{n \in \mathbb{N}}$ in $\mathcal{K}(\mu_l)$ with $\lim_{n \to \infty} V_{a,l}(S^a_n) = W_{a,l}(T_a)$ and converging to $T_a$ in $\mu_l$-measure.
Also, for every $b \in (a, k]$, write $\Delta_a^b = (Q_b - Q_a)/2$.
Then, by Jensen's inequality, and the continuity of $M_l(y)$ with respect to convergence in $\mu_l$-measure,
\begin{equation}
\begin{aligned}
V_b^* \left( y \right) - V_a^* \left( y \right)
&=
M_l \left( y \right) \left( T_b - T_a \right)
+
\ln \frac{\int_X \exp \left[ T_a \right] \mathrm{d} \mu_a}{\int_X \exp \left[ T_b \right] \mathrm{d} \mu_b} \\
&=
M_l \left( y \right) \left( T_b - T_a \right)
+
\lim_{n \to \infty}
\ln \frac{\int_X \exp \left[ S^a_n - T_b + \Delta_a^b \right] \exp \left[ T_b \right] \mathrm{d} \mu_b}{\int_X \exp \left[ T_b \right] \mathrm{d} \mu_b}
+
\ln \frac{N_b}{N_a} \\
&\ge
M_l \left( y \right) \left( T_b - T_a \right)
+
\liminf_{n \to \infty}
\frac{\int_X \left( S^a_n -T_b + \Delta_a^b \right) \exp \left[ T_b \right] \mathrm{d} \mu_b}{\int_X \exp \left[ T_b \right] \mathrm{d} \mu_b}
+
\ln \frac{N_b}{N_a} \\
&=
M_l \left( y \right) \left( T_b \right)
-
\frac{\int_X T_b \exp \left[ T_b \right] \mathrm{d} \mu_b}{\int_X \exp \left[ T_b \right] \mathrm{d} \mu_b}
+
\frac{\int_X \Delta_a^b \exp \left[ T_b \right] \mathrm{d} \mu_b}{\int_X \exp \left[ T_b \right] \mathrm{d} \mu_b}
+
\ln \frac{N_b}{N_a} \, .
\end{aligned}
\end{equation}
Thus, by \cref{eq:MlyRepresentation},
\begin{equation}
\begin{aligned}
\label{eq:VkAstDifferenceLowerBound}
V_b^* \left( y \right) - V_a^* \left( y \right)
\ge
\frac{\int_X \Delta_a^b \exp \left[ T_b \right] \mathrm{d} \mu_b}{\int_X \exp \left[ T_b \right] \mathrm{d} \mu_b}
+
\ln \frac{N_b}{N_a} \, .
\end{aligned}
\end{equation}
In the same spirit, we also have
\begin{equation}
\begin{aligned}
V_b^* \left( y \right) - V_a^* \left( y \right)
&=
M_l \left( y \right) \left( T_b - T_a \right)
-
\ln \frac{\int_X \exp \left[ T_b \right] \mathrm{d} \mu_b}{\int_X \exp \left[ T_a \right] \mathrm{d} \mu_a} \\
&=
M_l \left( y \right) \left( T_b - T_a \right)
-
\lim_{n \to \infty}
\ln \frac{\int_X \exp \left[ S^b_n - T_a - \Delta_a^b \right] \exp \left[ T_a \right] \mathrm{d} \mu_a}{\int_X \exp \left[ T_a \right] \mathrm{d} \mu_a}
+
\ln \frac{N_b}{N_a} \\
&\le
M_l \left( y \right) \left( T_b - T_a \right)
+
\limsup_{n \to \infty}
\frac{\int_X \left( T_a - S^b_n + \Delta_a^b \right) \exp \left[ T_a \right] \mathrm{d} \mu_a}{\int_X \exp \left[ T_a \right] \mathrm{d} \mu_a}
+
\ln \frac{N_b}{N_a} \\
&=
\frac{\int_X T_a \exp \left[ T_a \right] \mathrm{d} \mu_a}{\int_X \exp \left[ T_a \right] \mathrm{d} \mu_a}
-
M_l \left( y \right) \left( T_a \right)
+
\frac{\int_X \Delta_a^b \exp \left[ T_a \right] \mathrm{d} \mu_a}{\int_X \exp \left[ T_a \right] \mathrm{d} \mu_a}
+
\ln \frac{N_b}{N_a} \, .
\end{aligned}
\end{equation}
Using \cref{eq:MlyRepresentation} again,
\begin{equation}
\label{eq:VkAstDifferenceUpperBound}
V_b^* \left( y \right) - V_a^* \left( y \right)
\le
\frac{\int_X \Delta_a^b \exp \left[ T_a \right] \mathrm{d} \mu_a}{\int_X \exp \left[ T_a \right] \mathrm{d} \mu_a}
+
\ln \frac{N_b}{N_a} \, .
\end{equation}
For brevity, write $\sigma_a^b = Q_b - Q_a - \left( b - a \right) Q'_a$.
Focusing on the first term on the right-hand side, we now want to show that
\begin{equation}
\lim_{b \searrow a}
\int_X \frac{1}{2 \left( b - a \right)} \sigma_a^b \exp \left[ T_a \right] \mathrm{d} \mu_a
=
0 \, .
\end{equation}
To that end, we use the Orlicz-Hölder inequality as in the proof of \cref{lem:WPrimeSubdifferentialRepresentation} and obtain
\begin{equation}
\int_X \left| \frac{1}{2 \left( b - a \right)} \sigma_a^b \right| \exp \left[ T_a \right] \mathrm{d} \mu_a
\le
2 N^a_\Phi \left( \frac{1}{2 \left( b - a \right)} \sigma_a^b \right) N^a_\Psi \left( \exp \left[ T_a \right] \right) \, .
\end{equation}
It is again clear, that $N^a_\Psi \left( \exp \left[ T_a \right] \right)$ is finite.
Now, for any monotonically decreasing sequence $(b_n)_{n \in \mathbb{N}}$ in $(a,l)$ with $\inf_{n \in \mathbb{N}} b_n = a$ we invoke \cref{eq:QDerivativeCompatibility} to find a sequence $(r_n)_{n \in \mathbb{N}}$ in $[0, \infty)$ with $\lim_{n \to \infty} r_n = \infty$ such that
\begin{equation}
\lim_{n \to \infty} \int_X \exp \left[ \left| \frac{r_n}{2 \left( b_n - a \right)} \left( Q_{b_n} - Q_a \right) - \frac{r_n}{2} Q'_a \right| - \frac{1}{2} Q_a \right] \mathrm{d} \mu
=
\int_X \exp \left[ - \frac{1}{2} Q_a \right] \mathrm{d} \mu \, .
\end{equation}
Consequently,
\begin{equation}
\lim_{n \to \infty} \int_X \Phi \left( \frac{r_n}{2 \left( b_n - a \right)} \sigma_a^b \right) \mathrm{d} \mu_a
\le
1
\quad\text{i.e.}\quad
\lim_{n \to \infty} N^a_\Phi \left( \frac{1}{2 \left( b_n - a \right)} \sigma_a^b \right)
=
0 \, .
\end{equation}
Defining the function $f_y : [0, k], a \mapsto V_a^*(y) - \ln N_a$, we insert this into \cref{eq:VkAstDifferenceUpperBound} and obtain,
\begin{equation}
D^+ f_y \left( a \right)
\le
\frac{1}{2} \frac{\int_X Q'_a \exp \left[ T_a \right] \mathrm{d} \mu_a}{\int_X \exp \left[ T_a \right] \mathrm{d} \mu_a} \, .
\end{equation}
Proceeding analogously with \cref{eq:VkAstDifferenceLowerBound}, we also have
\begin{equation}
\label{eq:DiniMinusLowerLowerBound}
D_- f_y \left( b \right)
\ge
\frac{1}{2} \frac{\int_X Q'_b \exp \left[ T_b \right] \mathrm{d} \mu_b}{\int_X \exp \left[ T_b \right] \mathrm{d} \mu_b} \, .
\end{equation}
It is easy to see that $f_y$ is monotonically increasing, such that applying \cref{eq:QDerivativeBound} and the Orlicz-Hölder inequality, $D^+ f_y \left( a \right)$ is finite.
By \cref{cor:LogNkAbsoluteContinuity,thm:VAstContinuity}, $f_y$ is also continuous such that \cref{lem:FiniteDiniImpliesAbsoluteContinuity} applies and $f_y$ is absolutely continuous.
Using \cref{cor:LogNkAbsoluteContinuity} again, $V^*(y)$ is also absolutely continuous.
Finally, it follows from \cite[p. 45, Theorem 4.1]{src:Bruckner:Differentiation} that $D_- f_y(a) = D^+ f_y(a)$ except for at most countably many $a \in [0, k]$.
Hence, \cref{eq:VAstDerivative} follows.
\end{proof}
\begin{remark}
While the differential equation itself is quite remarkable, the mere absolute continuity instead of true differentiability is somewhat disappointing, especially because for the measures considered in \cite{src:Ziebell:RigorousFRG}, the author obtained ordinary differentiability.
As mentioned in the above proof, with the achieved bounds on $D_- f_y(a)$ and $D^+ f_y(a)$, it follows from \cite[p. 45, Theorem 4.1]{src:Bruckner:Differentiation} that $D_- f_y(a)$ and $D^+ f_y(a)$ coincide except for at most countably many points.
This suggests that a stronger result may be possible, although it has proven difficult to obtain useful bounds on the other two Dini derivatives.
A perhaps more fruitful approach could be to show that $f_y$ is \emph{nonangular} (see \cite[p. 142]{src:Bruckner:Differentiation}) and then true differentiability of $f_y$ would follow.
\end{remark}
The differential equation is intimately tied to the \textbf{effective average action} given as
\begin{equation}
\Gamma_k : \mathcal{A}(\mu) \to \overline{\mathbb{R}}, y \mapsto V_k^* \left( y \right) - \frac{1}{2} \underline{I_k} \left( y - w, y - w \right) \, ,
\end{equation}
where $y - w \in \mathcal{H}(\mu_{-w})$ by \cref{thm:AffineKernelContainedInShiftedKernel}.
\begin{corollary}
Let $0 \le k \le l$ and $y \in \mathcal{A}(\mu)$ such that $M_l(y) \in \mathrm{int}\, \mathrm{dom}\, V_{k,l}$ and for all $a \in [0, k]$, let $T_a \in \mathcal{L}(\mu_l)$ as in \cref{cor:WklMinimisation}.
Then, $\Gamma (y) : [0, k] \to \mathbb{R}, t \mapsto \Gamma_t(y)$ is absolutely continuous and for almost every $t \in [0, k]$,
\begin{equation}
\frac{\mathrm{d}}{\mathrm{d} t} \Gamma_t \left( y \right)
=
\frac{1}{2}
\frac{\int_X Q'_t \exp \left[ T_t \right] \mathrm{d} \mu_t}{\int_X \exp \left[ T_t \right] \mathrm{d} \mu_t}
-
\frac{1}{2} Q'_t \left( y \right)
-
\frac{1}{2} \int_X Q_t' \mathrm{d} \mu_t \, .
\end{equation}
\end{corollary}
\begin{proof}
It suffices to show that $t \mapsto Q_t(y)$ is absolutely continuous.
However, that follows immediately from \cref{thm:InnerProductOrderDiagonal,eq:QlPrimeOnKernelOfMuMinusW,lem:FiniteDiniImpliesAbsoluteContinuity}.
\end{proof}
Finally, using \cref{eq:QlPrimeOnKernelOfMuMinusW}, we shall write the right-hand side in terms of \enquote{second derivatives} of $W_{k,l}$.
Before elaborating on what that means, we introduce the abbreviation
\begin{equation}
\mu_k^{T} = \frac{\exp [ T ]}{\int_X \exp [ T ] \mathrm{d} \mu_k} \mu_k
\end{equation}
for $T \in \mathrm{dom}\, \partial W_{k,l}'$ and simply define
\begin{equation}
D^2 W_{k,l} \left( T \right) \left( S, S \right)
=
\int_X S^2 \mathrm{d} \mu_k^T
-
\left( \int_X S \mathrm{d} \mu_k^T \right)^2
\in [0, \infty] \, ,
\end{equation}
for all $S \in \mathcal{K}(\mu_l)$.
Note that the subtrahend is finite by \cref{lem:WPrimeSubdifferentialRepresentation}.
\begin{remark}
\label{rem:SecondDerivativeDiscussion}
Evidently, if $W_{k,l}$ is twice Gâteaux differentiable at $T$, the above formula gives the correct second derivative.
However, if $T \in \mathrm{dom}\, \partial W_{k,l}' \setminus \mathrm{int}\, \mathrm{dom}\, W_{k,l}'$, Gâteaux differentiability clearly fails.
Currently, the \enquote{second-order slice derivatives} considered in \cite{src:MohammadSoueycatt:SecondOrderSliceDerivative} appear to be the most general second-order derivatives of convex functions on Banach spaces that play well with convex conjugation.
However, it seems to the author that to obtain the correct definition of second-order derivative in this setting, one needs to take a \enquote{regularising limit} $T' \to T$ in an appropriate fashion, similar in spirit to the \enquote{limiting subdifferential} respectively the second-order subdifferentials considered in \cite{src:Mordukhovich:SecondOrder}.
However, in view of the asymmetric directional differentiability obtained in \cref{cor:ConjugateDiffability}, one may need to consider directional derivatives along a cone instead of along a vector space.
Finally, the resulting object $D^2 W_{k,l} ( T )$ needs to be considered in a very weak sense, because it is certainly possible that $D^2 W_{k,l} ( T ) ( S, S ) = \infty$ for some $S \in \mathcal{K}(\mu_l)$
The author is currently not aware of the existence of a suitable framework of second derivatives of convex functions that would make $D^2 W_{k,l}$ a derivative in a suitable sense.
\end{remark}
Clearly, $D^2 W_{k,l} \left( T \right)$ polarises to an inner product on the vector subspace $\mathcal{K}(\mu_l)^T$ of all $S$ for which the above is finite.
Moreover, $D^2 W_{k,l} \left( T \right)$ immediately complexifies to a sesquilinear inner product on $\mathcal{K}(\mu_l)^T_{\mathbb{C}}$ in the obvious way.
\begin{lemma}
\label{lem:QAndSecondWDerivative}
Let $0 \le k \le l$ and $y \in \mathcal{A}(\mu)$ such that $M_l(y) \in \mathrm{int}\, \mathrm{dom}\, V_{k,l}$ and $T_k \in \mathcal{L}(\mu_l)$ as in \cref{cor:WklMinimisation}.
For all $a \in A_k$, write $\omega^k_a = \phi^k_a + i \psi^k_a$ for some $\phi^k_a, \psi^k_a \in X^*$,
\begin{equation}
E_a = [ \phi^k_a - m_l( \phi^k_a ) ]_{\mu_l} \in \mathcal{K}(\mu_l)
\qquad\text{and}\qquad
F_a = [ \psi^k_a - m_l( \psi^k_a ) ]_{\mu_l} \in \mathcal{K}(\mu_l) \, ,
\end{equation}
as well as $U_a = E_a + i F_a$.
Then, $U_a \in \mathcal{K}(\mu_l)^{T_k}_{\mathbb{C}}$ for $\kappa_k$-almost every $a \in A_k$ and
\begin{equation}
\frac{1}{2} \int_X Q'_k \mathrm{d} \mu_k^{T_k}
-
\frac{1}{2} Q'_k \left( y \right)
=
\int_{A_k} D^2 W_{k,l}' \left( T_k \right) \left( U_a, U_a \right) \mathrm{d} \kappa_k \left( a \right) \, .
\end{equation}
\end{lemma}
\begin{proof}
First, note that for all $x \in X$,
\begin{equation}
Q_k' \left( x \right)
=
\int_{A_k} \left[ \phi^k_a \left( x - w \right)^2 + \psi^k_a \left( x - w \right)^2 \right] \mathrm{d} \kappa_k \left( a \right) \, .
\end{equation}
Focusing on just the first summand, we obtain,
\begin{equation}
\phi^k_a \left( x - w \right)^2
=
\left[ \left( \phi^k_a \left( x \right) - \int_X \phi^k_a \mathrm{d} \mu_l \right)
-
\left( \phi^k_a \left( w \right) - \int_X \phi^k_a \mathrm{d} \mu_l \right) \right]^2 \, .
\end{equation}
Obviously, the analogous statement holds for $\psi^k_a$ as well.
Hence, expanding via the binomial formula and using \cref{eq:MlyRepresentation}, the claim follows after some simple algebraic manipulations and Fubini's theorem.
\end{proof}
\begin{corollary}[Wetterich's Equation]
\label{cor:Wetterich}
Let $0 \le k \le l$ and $y \in \mathcal{A}(\mu)$ such that $M_l(y) \in \mathrm{int}\, \mathrm{dom}\, V_{k,l}^*$ and for all $t \in [0, k]$, let $T_t \in \mathcal{L}(\mu_l)$ as in \cref{cor:WklMinimisation}.
For almost all $t \in [0, k]$,
\begin{equation}
\frac{\mathrm{d}}{\mathrm{d} t} \Gamma_t \left( y \right)
=
\int_{A_k}
D^2 W_{t,l} \left( T_t \right) \left( U_a, U_a \right)
\mathrm{d} \kappa_k \left( a \right)
-
\frac{1}{2} \int_X Q_t' \mathrm{d} \mu_t \, .
\end{equation}
\end{corollary}
The missing step is now to relate $D^2 W_{t,l} ( T_t )$ to a suitable second derivative $D^2 W_{k,l}^*(M_l(y))$.
Then, Wetterich's equation, as given in \cite{src:Ziebell:RigorousFRG}, would follow.
As discussed in \cref{rem:SecondDerivativeDiscussion}, the author is at the moment not aware a way to establish this connection.
Fortunately, it should be noted that even though \cref{cor:ConjugateDiffability} just shows first-order directional differentiability of $W_{t,l}^*$ along a cone, this cone is actually very large as is shown in \cref{cor:SpanOfPositiveCone}.
Hence, the problem does not appear hopeless to the author.
This point is especially important, since one would want to be able to compute $D^2 W_{t,l} ( T_t ) ( U_a, U_a )$ from $\Gamma_t, Q_t$ and $M_l(y)$ alone - that is by just looking at $W_{t,l}^*$ restricted to $\mathrm{ran}\, M_l$ which is an affine subspace of $(\mathcal{L}(\mu_l), \overline{\tau_l})^*$.
With that regard, it may be useful to demand $w \in \mathcal{A}(\mu)$ as well.
Then $\mathcal{A}(\mu)$ can naturally be identified with a vector space, by defining $w$ to be the origin.
This would enable the common method of expanding the flow equation in a formal power series, thus bridging the gap to the usual theoretical physics setup.
\subsection{The Boundary Values}
\label{sec:BoundaryConditions}
Keeping the definitions from the last subsection, let us now look at boundary values of $\Gamma_k(y)$, that is $\Gamma_0(y)$ and $\lim_{k \to \infty} \Gamma_k(y)$.
In order to obtain meaningful results, we shall make the following assumptions which further clarify the role of $w$:
\begin{enumerate}
\setcounter{enumi}{3}
\item $I_0 = 0$ $\mu_{-w}$-almost everywhere.
\item\label{itm:JMonotonicity} There is a family $(r_k)_{k \ge 0}$ of positive real numbers with $\lim_{k \to \infty} r_k = \infty$ such that for every $k \ge 0$, there is some $K \ge 0$ such that for all $l \ge K$, $J_k \le J_l$, where $J_k = I_k / r_k^2$.
\item\label{itm:JBoundedness} For every sequence $k_n \nearrow \infty$ in $[0, \infty)$, $\sup_{n \in \mathbb{N}} \tilde{J}_{k_n}(x)$ is bounded $\mu_{-w}$-almost everywhere.
\item\label{itm:wInPSupport} For all $\epsilon > 0$,
\begin{equation}
\inf_{k \ge 0} \mu_{-w} \left( K^{J_k}_\epsilon(0) \right) > 0 \, .
\end{equation}
\end{enumerate}
\begin{corollary}
\label{cor:JConvergenceAndBallConvergence}
There is a $J \in \mathcal{I}(\mu_{-w})$ such that for every sequence $k_n \nearrow \infty$ in $[0, \infty)$ and $J_{k_n}$ converges to $J$.
Moreover, for all $\epsilon > 0$ and all $h \in \mathcal{H}(\mu_{-w})$,
\begin{equation}
\mu_{-w} \left( K^J_\epsilon(0) \right)
=
\inf_{k \ge 0} \mu_{-w} \left( K^{J_k}_\epsilon(0) \right)
=
\lim_{k \to \infty} \mu_{-w} \left( K^{J_k}_\epsilon(0) \right) \, .
\end{equation}
\end{corollary}
\begin{proof}
This follows immediately from \cref{itm:JMonotonicity,itm:JMonotonicity,thm:BoundedIncreasingInnerProductsConverge}.
\end{proof}
The following \namecref{lem:ExpMinusGamma} sheds light on what $\Gamma_k$ actually encodes.
\begin{lemma}
\label{lem:ExpMinusGamma}
Let $0 \le k$ and $y \in \mathcal{A}(\mu)$ such that $V_k^*(y) < \infty$.
Then,
\begin{equation}
\label{eq:ExpMinusGamma}
\exp \left[ - \Gamma_k \left( y \right) \right]
=
\inf_{\phi \in X^*} \frac{1}{N_k} \int_X \exp \left[ \phi \left( x \right) - \frac{1}{2} \tilde{I}_k \left( x \right)
\right] \mathrm{d} \mu_{-y}\left( x \right) \, .
\end{equation}
\end{lemma}
\begin{proof}
First, note that $y - w \in \mathcal{H}(\mu_{-w})$, by \cref{thm:AffineKernelContainedInShiftedKernel}.
Hence, \cref{prop:TranslationOfInnerProduct} applies and $I_k$ corresponds uniquely to an element in $\mathcal{I}(\mu_{-y})$ which we (somewhat sloppily) also denote by $I_k$.
Hence, $\tilde{I}_k$ is interpreted according to $\mathcal{I}(\mu_{-y})$.

Let $(B_n)_{n \in \mathbb{N}}$ be sequence as in \cref{def:MuRegularInnerProduct} that converges to $I_k$ in $\mu_{-w}^2$-measure and define
\begin{equation}
\label{eq:VknDefinition}
V_k^n : X^* \to \mathbb{R},
\phi
\mapsto
\ln \frac{\int_X \exp \left[ \phi \left( x \right)- \frac{1}{2} B_n \left( x - w, x - w \right) \right] \mathrm{d} \mu \left( x \right)}{\int_X \exp \left[ - \frac{1}{2} B_n \left( x - w, x - w \right) \right] \mathrm{d} \mu \left( x \right)} \, .
\end{equation}
Using an argument analogous to the one given in the proof of \cref{thm:VAstContinuity}, we then see that $\lim_{n \to \infty} (V_k^n)^*(y) = V_k^*(y)$.
Hence, writing $N_k^n$ for the denominator in \cref{eq:VknDefinition}, we have
\begin{equation}
\begin{aligned}
&\exp \left[ V_k^n \left( \phi \right) - y \left( \phi \right)  + \frac{1}{2} B_n \left( y - w, y - w \right) \right] \\
&=
\frac{1}{N_k^n} \int_X \exp \left[ \phi \left( x - y \right) - \frac{1}{2} B_n \left( x - w, x - w \right) + \frac{1}{2} B_n \left( y - w, y - w \right) \right] \mathrm{d} \mu\left( x \right) \\
&=
\frac{1}{N_k^n} \int_X \exp \left[ \phi \left( x \right)- \frac{1}{2} B_n \left( x + y - w, x + y - w \right)
+
\frac{1}{2} B_n \left( y - w, y - w \right) \right] \mathrm{d} \mu_{-y}\left( x \right) \\
&=
\frac{1}{N_k^n} \int_X \exp \bigg[ \phi \left( x \right) - B_n \left( y - w, x \right) - \frac{1}{2} B_n \left( x, x \right)
\bigg] \mathrm{d} \mu_{-y}\left( x \right) \, ,
\end{aligned}
\end{equation}
for all $\phi \in X^*$.
Because $B_n$ is continuous, it follows that
\begin{equation}
\exp \left[ - (V_k^n)^* \left( y \right) + \frac{1}{2} B_n \left( y - w, y - w \right) \right]
=
\inf_{\phi \in X^*}
\frac{1}{N_k^n} \int_X \exp \bigg[ \phi \left( x \right) - \frac{1}{2} B_n \left( x, x \right)
\bigg] \mathrm{d} \mu_{-y}\left( x \right) \, .
\end{equation}
Now, since $B_n(y-w, y-w)$ converges to $\underline{I_k}(y-w, y-w)$, we obtain
\begin{equation}
\begin{aligned}
\exp \left[ - \Gamma_k \left( y \right) \right]
&=
\lim_{n \to \infty} \exp \left[ - (V_k^n)^* \left( y \right) + \frac{1}{2} B_n \left( y - w, y - w \right) \right] \\
&=
\lim_{n \to \infty} 
\frac{1}{N_k^n}
\inf_{\phi \in X^*}
\int_X \exp \bigg[ \phi \left( x \right) - \frac{1}{2} B_n \left( x, x \right)
\bigg] \mathrm{d} \mu_{-y} \left( x \right) \\
&\ge
\frac{1}{N_k}
\inf_{\phi \in X^*}
\int_X \exp \bigg[ \phi \left( x \right) - \frac{1}{2} \tilde{I}_k \left( x \right)
\bigg] \mathrm{d} \mu_{-y} \left( x \right) \, .
\end{aligned}
\end{equation}
For the converse inequality, note that by the dominated convergence theorem,
\begin{equation}
\begin{aligned}
\lim_{n \to \infty} f_n \left( \phi \right)
&:=
\lim_{n \to \infty}
\int_X \exp \bigg[ \phi \left( x \right) - \frac{1}{2} B_n \left( x, x \right)
\bigg] \mathrm{d} \mu_{-y} \left( x \right) \\
&=
\int_X \exp \bigg[ \phi \left( x \right) - \frac{1}{2} \tilde{I}_k \left( x \right)
\bigg] \mathrm{d} \mu_{-y} \left( x \right)
=:
f \left( \phi \right) \, ,
\end{aligned}
\end{equation}
for all $\phi \in X^*$.
Furthermore, equipping $X^*$ with the discrete topology, it is obvious that $f$ is lower semicontinuous.
Hence, it follows from \cite[Proposition 5.7]{src:DalMaso:GammaConvergence}, that $f_n$ $\Gamma$-converges to $f$ and thus from \cite[Proposition 7.1]{src:DalMaso:GammaConvergence}, that
\begin{equation}
\inf_{\phi \in X^*} f \left( \phi \right)
\ge
\liminf_{n \to \infty} \inf_{\phi \in X^*} f_n \left( \phi \right) \, .
\end{equation}
\end{proof}
If one could remove the $\phi$ from \cref{eq:ExpMinusGamma}, the right-hand side would roughly encode the $\mu_{-y}$-measure of a $K^{I_k}$-ball around zero.
With that in mind, we call $y \in \mathcal{A}(\mu)$ \textbf{admissible}, if
\begin{equation}
\label{eq:AdmissibleDefinition}
\lim_{k \to \infty} \inf_{\phi \in X^*} \frac{\int_X \exp \left[ \phi \left( x \right) - \frac{1}{2} \tilde{I}_k \left( x \right) \right] \mathrm{d} \mu_{-y} \left( x \right)}{\int_X \exp \left[ - \frac{1}{2} \tilde{I}_k \left( x \right) \right] \mathrm{d} \mu_{-y} \left( x \right)}
=
1 \, .
\end{equation}
For a discussion on sufficient conditions ensuring this property, see \cref{subsec:AdmissiblePoints}.
Before turning to the final result of this section, we need another auxiliary object.
\begin{definition}
For every $k \ge 0$, let $\nu_k$ denote the measure on $[0, \infty)$ with density
\begin{equation}
s \mapsto \frac{r_k^2}{N_k} \mu \left( K^{I_k}_s ( w ) \right) s \exp \left[ - \frac{r_k^2}{2} s^2 \right]
\end{equation}
with respect to the Lebesgue measure.
Note that the density is measurable by the Layer cake representation.
\end{definition}
\begin{lemma}
\label{lem:NuKWeakConvergence}
$(\nu_k)_{k \ge 0}$ is a family of probability measures and $\nu_k$ converges weakly to the Dirac measure $\delta_0$ at $0$ as $k$ goes to infinity. 
\end{lemma}
\begin{proof}
For every $k \ge 0$, let $(Y_k, B_k)$ be a standard representation of $J_k$ and define the seminorm $p_k(x) = \sqrt{B_k(x,x)}$ on $Y_k$.
By \cref{prop:TranslationOfInnerProduct}, $p_k$ is Borel measurable, such that $x \mapsto p_k(x - w)$ is $\mu$-measurable.
Applying the substitution $t = \exp[-r_k^2 s^2 / 2]$ and the Layer cake representation,
\begin{equation}
\label{eq:NuKTotalMass}
\begin{aligned}
\nu_k \left( \left[ 0, \infty \right) \right)
&=
\frac{1}{N_k}
\int_0^1
\mu \left( \left\{ x \in Y_k \middle| \exp \left[ - \frac{r_k^2}{2} p_k \left( x - w \right)^2 \right] \right\} > t \right)
\mathrm{d} t \\
&=
\frac{1}{N_k} \int_X \exp \left[ - \frac{r_k^2}{2} p_k \left( x - w \right)^2 \right] \mathrm{d} \mu
=
1 \, .
\end{aligned}
\end{equation}
Now, let $\epsilon > 0$ and use the same technique to see that
\begin{equation}
\nu_k \left( \left[ \epsilon, \infty \right) \right)
=
\frac{1}{N_k} \int_{p_k(x-w) \ge \epsilon} \exp \left[ - \frac{1}{2} r_k^2 p_k \left( x - w \right)^2 \right] \mathrm{d} \mu \\
\le
\frac{1}{N_k} \exp \left[ - \frac{r_k^2}{2} \epsilon^2 \right] \, .
\end{equation}
At the same time,
\begin{equation}
N_k
\ge
\int_{K^{B_k}_{\epsilon/2}(w)} \exp \left[ - \frac{1}{2} Q_k \right] \mathrm{d} \mu
\ge
\exp \left[ - \frac{r_k^2}{8} \epsilon^2 \right]
\mu \left( K^{B_k}_{\epsilon/2}(w) \right) \, .
\end{equation}
Furthermore, by \cref{cor:JConvergenceAndBallConvergence},
\begin{equation}
\mu( K^{B_k}_{\epsilon/2}(w) )
=
\mu( K^{J_k}_{\epsilon/2}(w) )
\ge
\mu( K^J_{\epsilon/2}(w) ) > 0 \, .
\end{equation}
Thus $\lim_{k \to \infty} \nu_k ( [ \epsilon, \infty ) ) = 0$.
Consequently, $\lim_{k \to \infty} \nu_k ( U ) = 1$ for every open neighbourhood $U$ of zero.
For arbitrary open sets $V \subseteq [0, \infty)$ we then have, $\liminf_{k \to \infty} \nu_k ( V ) \ge \delta_0 ( V )$.
Hence, the weak convergence follows (see e.g. \cite[Theorem 8.1.4]{src:Bogachev:MeasuresOnTopologicalSpaces}).
\end{proof}
\begin{definition}[{\cite[Chapter 4.7]{src:Bogachev:GaussianMeasures}}]
Let $a, b \in \mathcal{H}(\mu_{-w})$ and define the \textbf{Onsager-Machlup function}
\begin{equation}
\mathcal{F} \left( a, b \right)
=
\lim_{s \searrow 0} \ln \frac{\mu \left( K^J_s(a) \right)}{\mu \left( K^J_s(b) \right)}
\in
\overline{\mathbb{R}}
\end{equation}
whenever the limit exists in $\overline{\mathbb{R}}$ where we also set $\ln(0) = -\infty$ and $\ln(\infty) = \infty$ and choose the convention $c/0 = \infty$ for all $c > 0$ and leave $0/0$ undefined.
Similarly, we define the generalisation
\begin{equation}
\overline{\mathcal{F}} \left( a, b \right)
=
\lim_{(k ,s) \to (\infty, 0)} \ln \frac{\mu \left( K^{J_k}_s(a) \right)}{\mu \left( K^{J_k}_s(b) \right)}
\in
\overline{\mathbb{R}}
\end{equation}
whenever it exists.
\end{definition}
With \cref{cor:JConvergenceAndBallConvergence} one may verify, that the existence of $\overline{\mathcal{F}}$ is stronger than that of $\mathcal{F}$.
\begin{corollary}
Whenever, $\overline{\mathcal{F}}(a,b)$ exists in $\overline{\mathbb{R}}$ and
\begin{equation}
\label{eq:OnsagerMachlupExistenceCondition}
\forall s > 0 : \mu \left( K^J_s(a) \right) > 0
\qquad\text{or}\qquad
\forall s > 0 : \mu \left( K^J_s(b) \right) > 0 \, ,
\end{equation}
then $\mathcal{F}(a,b)$ exists in $\overline{\mathbb{R}}$ and has the same value.
If neither holds, $\mathcal{F}(a,b)$ is undefined.
\end{corollary}

\begin{theorem}
Suppose that $y \in \mathcal{A}(\mu)$ is an admissible point and that $\overline{\mathcal{F}} \left( w, y \right)$ exists.
Then,
\begin{equation}
\lim_{k \to \infty} \Gamma_k \left( y \right)
=
\overline{\mathcal{F}} \left( w, y \right) \, .
\end{equation}
\end{theorem}
\begin{proof}
Since $y$ is admissible, it follows from \cref{lem:ExpMinusGamma}, that
\begin{equation}
\lim_{k \to \infty} \Gamma_k \left( y \right)
=
\lim_{k \to \infty} \ln \frac{N_k}{\int_X \exp \left[ - \frac{1}{2} \tilde{I}_k \right] \mathrm{d} \mu_{-y}} \, ,
\end{equation}
if the limit on the right-hand side exists in $\overline{\mathbb{R}}$.
Using the layer cake representation, we find
\begin{equation}
\frac{1}{N_k} \int_X \exp \left[ - \frac{1}{2} \tilde{I}_k \right] \mathrm{d} \mu_{-y}
=
\int_0^\infty \frac{ \mu \left( K^{J_k}_s(y) \right)}{\mu \left( K^{J_k}_s(w) \right)} \mathrm{d} \nu_k \left( s \right) \, .
\end{equation}
Fixing any $\epsilon > 0$, the integrand is bounded on $(\epsilon, \infty)$ such that only the integral over $[0, \epsilon]$ contributes to the limit as $k$ goes to infinity.
Hence, because $\overline{\mathcal{F}}(w, y)$ exists and $\nu_k$ converges weakly to $\delta_0$, we find
\begin{equation}
\lim_{k \to \infty} \frac{1}{N_k} \int_X \exp \left[ - \frac{1}{2} \tilde{I}_k \right] \mathrm{d} \mu_{-y}
=
\exp \left[ - \overline{\mathcal{F}}(w, y) \right]
\end{equation}
and the claim follows.
\end{proof}
\subsection{Admissible Points}
\label{subsec:AdmissiblePoints}
The most simple criterion for \cref{eq:AdmissibleDefinition} to hold, is that $\mu_{-y}$ is symmetric.
Since $\tilde{I}_k$ is also symmetric $\mu_{-y}$-almost everywhere, the infimum is attained at $0$ and the left-hand side is equal to $1$ at every $k \ge 0$.
However, that can obviously just hold at most for a single $y \in \mathcal{A}(\mu)$.
The idea can however, be generalised to to an \enquote{approximate symmetry}.
To that end, let $\overline{\mu_{-y}}$ denote the pushforward measure of $\mu_{-y}$ along the function $x \mapsto - x$.
\begin{theorem}
Suppose that for all $s > 0$, $\mu(K^J_s(y)) > 0$ and for every $\epsilon > 0$, there are some $a \ge 0$ and $\delta > 0$ such that for some standard representation $(Y_a, B_a)$ of $J_a$,
\begin{equation}
\frac{\mathrm{d} \mu_{-y} \wedge \overline{\mu_{-y}}}{\mathrm{d} \mu_{-y}} \left( x \right)
\ge
1 - \epsilon \, ,
\end{equation}
for $\mu_{-y}$-almost every $x \in K^{J_a}_\delta(0)$, where $\wedge$ denotes the greatest lower bound in the lattice of measures.
Then $y$ is admissible.
\end{theorem}
\begin{proof}
Let $\epsilon > 0$ and $a, \delta$ as above.
Then, by the symmetry of $\mu_{-y} \wedge \overline{\mu_{-y}}$,
\begin{equation}
\begin{aligned}
\inf_{\phi \in X^*} \int_X \exp \left[ \phi - \frac{1}{2} \tilde{I}_k \right] \mathrm{d} \mu_{-y}
&\ge
\int_X \exp \left[ - \frac{1}{2} \tilde{I}_k \right] \mathrm{d} \left( \mu_{-y} \wedge \overline{\mu_{-y}} \right) \\
&\ge
\int_{K^{B_a}_\delta(0)} \exp \left[ - \frac{1}{2} \tilde{I}_k \right] \frac{\mathrm{d} \mu_{-y} \wedge \overline{\mu_{-y}}}{\mathrm{d} \mu_{-y}} \mathrm{d} \mu_{-y} \\
&\ge
\left( 1 - \epsilon \right) \int_{K^{B_a}_\delta(0)} \exp \left[ - \frac{1}{2} \tilde{I}_k \right] \mathrm{d} \mu_{-y} \, .
\end{aligned}
\end{equation}
Hence,
\begin{equation}
\liminf_{k \to \infty} \inf_{\phi \in X^*}
\frac{\int_X \exp \left[ \phi - \frac{1}{2} \tilde{I}_k \right] \mathrm{d} \mu_{-y}}{\int_{K^{B_a}_\delta(0)} \exp \left[ - \frac{1}{2} \tilde{I}_k \right] \mathrm{d} \mu_{-y}}
\ge
1 - \epsilon \, ,
\end{equation}
Now, because $\mu(K^J_s(y)) > 0$ for all $s > 0$, a proof along the lines of \cref{lem:NuKWeakConvergence}, shows that
\begin{equation}
\lim_{k \to \infty} \frac{\int_{K^{B_a}_\delta(0)} \exp \left[ - \frac{1}{2} \tilde{I}_k \right] \mathrm{d} \mu_{-y}}{\int_X \exp \left[ - \frac{1}{2} \tilde{I}_k \right] \mathrm{d} \mu_{-y}}
=
1 \, .
\end{equation}
Consequently,
\begin{equation}
\liminf_{k \to \infty} \inf_{\phi \in X^*}
\frac{\int_X \exp \left[ \phi - \frac{1}{2} \tilde{I}_k \right] \mathrm{d} \mu_{-y}}{\int_X \exp \left[ - \frac{1}{2} \tilde{I}_k \right] \mathrm{d} \mu_{-y}}
\ge
1 - \epsilon \,
\end{equation}
and because $\epsilon > 0$ was arbitrary, the claim follows.
\end{proof}
Unfortunately, the approximate symmetry is also a very strong requirement.
A much better result can be given when $\mu$ is equivalent to a centred Radon Gaußian measure with a suitable density.
Then all $y \in \mathcal{A}(\mu)$ are admissible as was shown in special cases in \cite{src:Ziebell:RigorousFRG}.
Unfortunately, less restrictive sufficient conditions are at present not known to the author.
\section{Conclusion}
It has been shown that Wetterich's equation indeed holds for almost every flow parameter $k \ge 0$, provided that the point $y \in \mathcal{A}(\mu)$ lies in the interior of the domain of the auxillary function $V_{k,l}^*$.
The interiority condition is clearly of no surprise, because subdifferentials are often well-behaved on these interiors only.
Consequently, the setup is most applicable to measures whose moment-generating functions grow faster than $\exp[ C \Vert \cdot \Vert_{L^1(\mu_l)} ]$ for all $C > 0$.
Then $\mathrm{dom}\, V_{k,l}^*$ becomes the whole space and all $y \in \mathcal{A}(\mu)$ lie in the corresponding interior.
A setback is the mere absolute continuity of $k \mapsto \Gamma_k(y)$ on compact intervals.
It would certainly be desirable to obtain (continuous) differentiability.
Furthermore, it is unclear whether
\begin{equation}
\lim_{k \to \infty} \Gamma_k(y)
=
\int_0^\infty \frac{\mathrm{d}}{\mathrm{d} k}\Gamma_k \left( y \right) \mathrm{d} k
\end{equation}
holds as a Lebesgue integral.
It has not been shown that
\begin{equation}
D^2 W_{k,l} \left( T_k \right) \left( U_a, U_a \right))
=
U_a \left[ \left( D^2 \Gamma_k \left( y \right) + \underline{I_k} \right)^{-1} \left( U_a \right) \right]
\end{equation}
holds as it does in the cases considered in \cite{src:Ziebell:RigorousFRG}.
This is due to a missing suitable definition of second derivative of $W_{k,l}$ that is compatible with convex conjugates and was discussed in \cref{rem:SecondDerivativeDiscussion}.
\cref{cor:Wetterich} is also an impressive incarnation of the implicit function theorem.
At a point $y \in \mathcal{A}(\mu)$ in the considered interior, one has
\begin{equation}
\Gamma_k \left( y \right)
=
M_l \left( y \right) \left( \left( \partial W_{k,l}' \right)^{-1} \left( y \right) \right)
-
W_{k,l} \left( \left( \partial W_{k,l}' \right)^{-1} \left( y \right) \right)
-
\frac{1}{2} \underline{I_k} \left( y - w, y - w \right) \, ,
\end{equation}
while Wetterich's equation gives the corresponding derivative with respect to $k$ in terms of $D^2 W_{k,l}( ( \partial W_{k,l}' )^{-1} ( y ) )$.
This may provide a hint, to the existence of a generalised implicit function theorem in convex analysis.

The asymptotic boundary condition of $\Gamma_k(y)$ is a generalised version of the Onsager-Machlup function with respect to the metrics induced by the family $(J_k)_{k \ge 0}$ of measurable inner products, provided that $y$ is \enquote{admissible}.
Rather restrictive sufficient conditions for admissibility were given in \cref{subsec:AdmissiblePoints} while more widely applicable ones are at present unknown to the author.

The measurable inner products introduced in \cref{sec:MeasurableInnerProducts} are of particular interest on their own, being bilinear analogoues of the linear Lusin measurable functionals considered in \cite{src:Slowikowski:PreSupports}.
With a separability condition, they are, in fact, examples of measurable functions on a product space with an almost-everywhere well-defined restriction to the diagonal.
\section{Appendix}
In order to understand the derivative in \cref{cor:ConjugateDiffability}, it is useful to characterise $(\mathcal{L}(\mu_l), \overline{p_l})^*$.
\begin{lemma}
For every $\alpha \in (\mathcal{L}(\mu_l), \overline{p_l})^*$, there exists some nonnegative $f_\alpha \in L^\infty(\mu_l)$ such that
\begin{equation}
\alpha \left( T \right)
=
\int_X f_\alpha T \mathrm{d} \mu_l
\end{equation}
for all $T \in \mathcal{L}(\mu_l)$.
Conversely, for every nonnegative $f \in L^\infty(\mu_l)$, the mapping
\begin{equation}
T \mapsto \int_X f T \mathrm{d} \mu_l
\end{equation}
is an element of $(\mathcal{L}(\mu_l), \overline{p_l})^*$.
\end{lemma}
\begin{proof}
The space $L^1(\mu_l)$ with its standard partial order is an ordered topological vector space with positive cone given by
\begin{equation}
P = \left\{ f \in L^1(\mu_l) : f \ge 0 \right\} \, .
\end{equation}
For every $\alpha \in (\mathcal{L}(\mu_l), \overline{p_l})^*$, there is some $C > 0$ such that $\alpha(T) \le C \overline{p_l}(T)$ for all $T \in Y$.
Let $U \subseteq L^1(\mu_l)$ denote the closed unit ball, and fix any $f \in \mathcal{L}(\mu_l) \cap [ U - P ]$.
Then, there is some $p \in P$ such that $f + p \in U$ and
\begin{equation}
\alpha \left( f \right)
\le
C \int_X \max \left\{ f, 0 \right\} \mathrm{d} \mu_l
\le
C \int_X \max \left\{ f + p, 0 \right\} \mathrm{d} \mu_l
\le
C \, .
\end{equation}
Hence, \cite[p. 227, Theorem 5.4]{src:SchaeferWolff:TVS} applies and there is some continuous nonnegative linear functional on $L^1(\mu_l)$ whose restriction to $\mathcal{L}(\mu_l)$ coincides with $\alpha$.
Such an extension is given by integration against some nonnegative $f_\alpha \in L^\infty(\mu_l)$.
For the converse, let $f \in L^\infty(\mu_l)$ be nonnegative and $T \in \mathcal{L}(\mu_l)$.
Then,
\begin{equation}
\int_X f T \mathrm{d} \mu_l
\le
\int_X f \max\left\{ T, 0 \right\} \mathrm{d} \mu_l
\le
\left\Vert f \right\Vert_{L^\infty(f)} \overline{p_l} \left( T \right) \, .
\end{equation}
\end{proof}
\begin{corollary}
\label{cor:SpanOfPositiveCone}
$( \mathcal{L}(\mu_l), L^1(\mu_l) )^* = (\mathcal{L}(\mu_l), \overline{p_l})^* - (\mathcal{L}(\mu_l), \overline{p_l})^*$.
\end{corollary}
\appendix
\printbibliography
\end{document}